\documentclass[10pt]{IEEEtran}

\usepackage{amssymb,amstext,amsmath,amsthm}
\usepackage{float}
\usepackage{esint}
\usepackage{url}
\usepackage[usenames,dvipsnames]{color}
\usepackage{verbatim}
\usepackage{enumerate}
\usepackage{ifthen}
\usepackage{graphics}
\usepackage{subfigure}
\usepackage{mathrsfs}
\usepackage{amsbsy}
\usepackage{graphicx,epstopdf}
\usepackage{cite}
\usepackage{algorithm}
\usepackage{algorithmic}
\usepackage{dcolumn}
\usepackage{subeqnarray}
\usepackage{multirow}

\newtheorem{theorem}{Theorem}
\newtheorem{lemma}{Lemma}

\newtheorem{corollary}{Corollary}
\newtheorem{definition}{Definition}
\newtheorem{claim}{Claim}
\newtheorem{proposition}{Proposition}
\newcommand{\hE}{\mathcal{E}}
\newcommand{\hN}{\mathcal{N}}
\newcommand{\hF}{\mathcal{F}}
\newcommand{\hP}{\mathcal{P}}

\newcommand{\hL}{\mathcal{L}}

\newcommand{\ii}{\textbf{i}}
\newcommand{\beq}{\begin{equation}}
\newcommand{\eeq}{\end{equation}}
\newcommand{\bq}{\begin{eqnarray}}
\newcommand{\eq}{\end{eqnarray}}
\newcommand{\bqn}{\begin{eqnarray*}}
\newcommand{\eqn}{\end{eqnarray*}}
\newcommand{\bee}{\begin{enumerate}}
\newcommand{\eee}{\end{enumerate}}
\newcommand{\bi}{\begin{itemize}}
\newcommand{\ei}{\end{itemize}}
\newcommand{\re}{\mathrm{Re}}
\newcommand{\im}{\mathrm{Im}}
\newcommand{\conv}{\mathrm{Conv}}
\newcommand{\rank}{\mathrm{Rank}}
\newcommand{\eqdef}{:=}

\newboolean{showcomments}
\setboolean{showcomments}{true}
\newcommand{\lgan}[1]{\ifthenelse{\boolean{showcomments}}
{ \textcolor{red}{(Lingwen says: #1)} } {} }
\newcommand{\lina}[1]{\ifthenelse{\boolean{showcomments}}
{ \textcolor{green}{(Lina says:  #1)}}{}}
\newcommand{\slow}[1]{\ifthenelse{\boolean{showcomments}}
{ \textcolor{blue}{(Steven says:  #1)}}{}}
\newcommand{\addcite}[0]{\ifthenelse{\boolean{showcomments}}
{ \textcolor{red}{(addcite)}}{}}
\newcommand{\addcites}[0]{\ifthenelse{\boolean{showcomments}}
{ \textcolor{red}{(addcite(s))}}{}}
\newcommand{\addref}[0]{\ifthenelse{\boolean{showcomments}}
{ \textcolor{Blue}{(addref)}}{}}
\newcommand{\todo}[1]{\ifthenelse{\boolean{showcomments}}
{ \textcolor{red}{(To do: #1)}} {} }

\begin{document}

\title{Exact Convex Relaxation of Optimal Power Flow in Tree Networks}
\author{Lingwen Gan, Na Li, Ufuk Topcu, and Steven H. Low
        \thanks{This work was supported by NSF NetSE grant CNS 0911041, ARPA-E grant DE-AR0000226, Southern California Edison, National Science Council of Taiwan, R.O.C, grant NSC 101-3113-P-008-001, Resnick Institute, Okawa Foundation, NSF CNS 1312390, DoE grant DE-EE000289, and AFOSR award number FA9550-12-1-0302.

         Lingwen Gan, Na Li, and Steven H. Low are with the Engineering and Applied Science Department, California Institute of Technology, Pasadena, CA 91125 USA (e-mail: lgan@caltech.edu; nali@caltech.edu; slow@caltech.edu).

         Ufuk Topcu is with the Electrical and Systems Engineering Department, University of Pennsylvania, Philadelphia, PA 19104 USA (e-mail: utopcu@seas.upenn.edu).}
}
\maketitle

\begin{abstract}
The optimal power flow (OPF) problem seeks to control power generation/demand to optimize certain objectives such as minimizing the generation cost or power loss in the network. It is becoming increasingly important for distribution networks, which are tree networks, due to the emergence of distributed generation and controllable loads. In this paper, we study the OPF problem in tree networks. The OPF problem is nonconvex. We prove that after a ``small'' modification to the OPF problem, its global optimum can be recovered via a second-order cone programming (SOCP) relaxation, under a ``mild'' condition that can be checked apriori. Empirical studies justify that the modification to OPF is ``small'' and that the ``mild'' condition holds for the IEEE 13-bus distribution network and two real-world networks with high penetration of distributed generation.
\end{abstract} 
\section{Introduction}
The optimal power flow (OPF) problem seeks to control power generation/demand to optimize certain objectives such as minimizing the generation cost or power loss in the network. It is proposed by Carpentier in 1962 \cite{OPF} and has been one of the fundamental problems in power system operation ever since.

The OPF problem is becoming increasingly important for distribution networks, which are tree networks, due to the emergence of distributed generation \cite{DG} (e.g., rooftop solar panels) and controllable loads \cite{EV} (e.g., electric vehicles). Distributed generation is difficult to predict, calling the traditional control strategy of ``generation follows demand'' into question. Meanwhile, controllable loads provide significant potential to compensate for the randomness in distributed generation \cite{DR}. To integrate distributed generation and realize the potential of controllable loads, solving the OPF problem in real time for tree networks is inevitable.

The OPF problem is difficult to solve due to its nonconvex power flow constraints. There are in general three ways to deal with this challenge: (i) linearize the power flow constraints; (ii) look for local optima; and (iii) convexify power flow constraints, which are described in turn.

The power flow constraints can be well approximated by some linear constraints in transmission networks, and then the OPF problem reduces to a linear programming \cite{Stott74,Alsac90,Stott09}. This method is widely used in practice for transmission networks, but does not apply to distribution networks, nor problems that consider reactive power flow or voltage deviations explicitly.

Various algorithms have been proposed to find local optima of the OPF problem, e.g., successive linear/quadratic programming \cite{Contaxis1986}, trust-region based methods \cite{Min2005,Sousa2011}, Lagrangian Newton method \cite{Baptista2005}, and interior-point methods \cite{Torres1998,Jabr2003,Capitanescu2007}. However, a local optimum can be highly suboptimal.

Convexification methods are the focus of this paper. It is proposed in \cite{Bai08,Bai09,Javad12} to transform the nonconvex power flow constraints into linear constraints on a rank-one positive-semidefinite matrix, and then remove the rank-one constraint to obtain a semidefinite programming (SDP) relaxation. If the solution one obtains by solving the SDP relaxation is of rank one, then a global optimum of OPF can be recovered. In this case, we say that the SDP relaxation is {\it exact}. Strikingly, it is claimed in \cite{Javad12} that the SDP relaxation is exact for the IEEE 14-, 30-, 57-, and 118-bus test networks, highlighting the potential of convexification methods.

Another type of convex relaxations, i.e., second-order cone programming (SOCP) relaxations, have also been proposed to solve the OPF problem \cite{Masoud11,Taylor11,Jabr06}. While having a much lower computational complexity than the SDP relaxation, the SOCP relaxation is exact if and only if the SDP relaxation is exact for tree networks \cite{sojoudi2012physics,Bose_equivalence}. Hence, we focus on the SOCP relaxation in this paper, in particular the SOCP relaxation proposed in \cite{sojoudi2012physics}.

Up to date, sufficient conditions that have been derived for the exactness of the SOCP relaxation do not hold in practice \cite{Zhang12,Bose12,Masoud11,lam2012distributed}. For example, the conditions in \cite{Zhang12,Bose12,Masoud11} require some/all buses to be able to draw infinite power, and the condition in \cite{lam2012distributed} requires a fixed voltage at every bus.

\subsection*{Summary of contributions}
The goal of this paper is to prove that, the SOCP relaxation is exact under a mild condition, for a modified OPF problem. The condition holds for all test networks considered in this paper. The modified OPF problem has the same objective function as the OPF problem, but a slightly smaller feasible set. In particular, contributions of this paper are threefold.

First, we prove that {\it under Condition C1 (Lemma \ref{lemma: exact}), the SOCP relaxation is exact if its solutions lie in some region $\mathcal{S}$}. Condition C1 can be checked apriori, and holds for the IEEE 13-bus distribution network and two real-world  networks with high penetration of distributed generation. The proof of Condition C1 explores the feasible set of the SOCP relaxation: for any feasible point $w$ of the SOCP relaxation that is (in $\mathcal{S}$ but) infeasible for the OPF problem, one can find another feasible point $w'$ of the SOCP relaxation with a smaller objective value (if Condition C1 holds). Hence, optimal solutions of the SOCP relaxation, if in $\mathcal{S}$, must be feasible for the OPF problem.

Second, we {\it modify the OPF problem by intersecting its feasible set with $\mathcal{S}$}. This modification is necessary since otherwise examples exist where the SOCP relaxation is not exact. Remarkably, with this modification, only feasible points that are ``close'' to the voltage regulation upper bounds are eliminated, and the SOCP relaxation is exact under Condition C1. Empirical studies justify that the modification to the OPF problem is ``small'' for the IEEE 13-bus distribution network and two real-world networks with high penetration of distributed generation.

Third, we prove that {\it the SOCP relaxation has at most one solution if it is exact}. In this case, any convex programming solver gives the same solution. 
\section{The optimal power flow problem}\label{sec: opf}
This paper studies the optimal power flow (OPF) problem in distribution networks, which includes Volt/VAR control and demand response. In the following we present a model of this scenario that serves as the basis for our analysis. The model incorporates nonlinear power flow physical laws, considers a variety of controllable devices including distributed generators, inverters, controllable loads, and shunt capacitors, and allows for a wide range of control objectives such as minimizing the power loss or generation cost, which are described in turn.

\subsection{Power flow model}
A distribution network is composed of buses and lines connecting these buses, and has a tree topology.

There is a substation in a distribution network, which has a fixed voltage and a flexible power injection for power balance. Index the substation bus by 0 and the other buses by $1,\ldots,n$. Let $\hN\eqdef\{0,\ldots,n\}$ denote the set of all buses and $\hN^+:=\{1,\ldots,n\}$ denote the set of all non-substation buses. Each line connects an ordered pair $(i,j)$ of buses where bus $j$ is between bus $i$ and bus 0. Let $\hE$ denote the set of all lines and abbreviate $(i,j)\in \hE$ by $i\rightarrow j$. If $i\rightarrow j$ or $j\rightarrow i$, denote $i\sim j$; otherwise denote $i\nsim j$.

For each bus $i\in \hN$, let $V_i$ denote its voltage and $I_i$ denote its current injection. Specifically, the substation voltage, $V_0$, is given and fixed. Let $s_i= p_i+\ii q_i$ denote the power injection of bus $i$ where $p_i$ and $q_i$ denote its real and reactive power injections respectively. Specifically, $s_0$ is the power that the substation draws from the transmission network for power balance. Let $\hP_i$ denote the path (a collection of buses in $\hN$ and lines in $\hE$) from bus $i$ to bus 0.

For each line $i\sim j$, let $y_{ij}= g_{ij}-\ii b_{ij}$ denote its admittance and $z_{ij}= r_{ij}+\ii x_{ij}$ denote its impedance, then $y_{ij}z_{ij}=1$.

	\begin{figure}[h]
     	\centering
     	\includegraphics[scale=0.36]{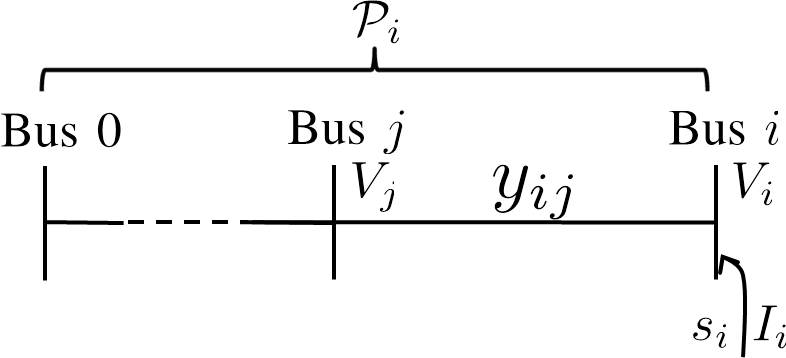}
      	\caption{Some of the notations.}
      	\label{fig: notation}
	\end{figure}

Some of the notations are summarized in Fig. \ref{fig: notation}. Further, we use a letter without subscripts to denote a vector of the corresponding quantity, e.g., $V=(V_1,\ldots,V_n)$, $y=(y_{ij},i\sim j)$. Note that subscript 0 is not included in nodal variables.

Given the network graph $(\hN, \hE)$, the admittance $y$, and the substation voltage $V_0$, then the other variables $(s,V,I,s_0)$ are described by the following physical laws.
\begin{itemize}
\item Current balance and Ohm's law:
$$I_i = \sum_{j:\,j\sim i}y_{ij} (V_i - V_j) ,  \qquad  i\in \hN;$$
\item Power balance:
$$s_i =  V_i I_i^*,  \qquad i\in \hN.$$
\end{itemize}
If we are only interested in voltages and power, then the two sets of equations can be combined into a single one
	\begin{equation}\label{PF}
	s_i = V_i\sum_{j:\,j\sim i}(V_i^*-V_j^*)y_{ij}^*, \qquad i\in \hN.
	\end{equation}
In this paper, we use \eqref{PF} to model the power flow.

\subsection{Controllable devices and control objective}
\label{sec: control elements}
Controllable devices in a distribution network include distributed generators; inverters that connect distributed generators to the grid; controllable loads like electric vehicles and smart appliances; and shunt capacitors.

Real and reactive power generation/demand of these devices can be controlled to achieve certain objectives. For example, in Volt/VAR control, reactive power injection of the inverters and shunt capacitors are controlled to regulate the voltages; in demand response, real power consumption of controllable loads are reduced or shifted in response to power supply conditions. Mathematically, power injection $s$ is the control variable, after specifying which the other variables $V$ and $s_0$ are determined by \eqref{PF}.

Constraints on the power injection $s_i$ of a bus $i\in \hN^+$ is captured by some feasible power injection set $\mathcal{S}_i$, i.e.,
	\begin{equation}\label{constraint s}
	s_i \in \mathcal{S}_i, \qquad i\in \hN^+.
	\end{equation}
The set $\mathcal{S}_i$ for some control devices are as follows.
\begin{itemize}
\item If bus $i$ has a shunt capacitor with nameplate capacity $\overline{q}_i$, then
	$$\mathcal{S}_i = \{s\in\mathbb{C}~|~\re(s)=0, ~\im(s)=0\text{ or }\overline{q}_i\}.$$
\item If bus $i$ has a solar panel with generation capacity $\overline{p}_i$, and an inverter with nameplate capacity $\overline{s}_i$, then
	$$\mathcal{S}_i = \{s\in\mathbb{C}~|~0\leq \re(s)\leq \overline{p}_i, ~|s|\leq \overline{s}_i \}.$$
\item If bus $i$ has a controllable load with constant power factor $\eta$, whose real power consumption can vary continuously from $-\overline{p}_i$ to $-\underline{p}_i$, then
	\begin{equation*}\mathcal{S}_i = \left\{s\in\mathbb{C}~\left|~
    \begin{aligned}
    & \underline{p}_i\leq \re(s)\leq \overline{p}_i,\\
    & \im(s)=\sqrt{1-\eta^2}\re(s)/\eta
    \end{aligned}
    \right. \right\}.
    \end{equation*}
\end{itemize}

The control objective in a distribution network is twofold. The first one is regulating the voltages within a range. This is captured by externally specified voltage lower and upper bounds $\underline{V}_i$ and  $\overline{V}_i$, i.e.,
    \begin{equation}\label{constraint v}
	\underline{V}_i \leq |V_i| \leq \overline{V}_i, \qquad i\in \hN^+.
	\end{equation}
For example, if 5\% voltage deviation from the nominal value is allowed, then $0.95\leq |V_i|\leq 1.05$ \cite{ANSI_C84}.

The second objective is minimizing the generation cost. For $i\in\hN,$ let $f_i(\re(s_i))$ denote the generation cost of bus $i$ where $f_i$ is a real-valued function defined on $\mathbb{R}$. Then, generation cost $C$ in the network is
    \begin{equation}\label{generation cost}
	C(s,s_0) = \sum_{i\in\hN} f_i(\re(s_i)).
	\end{equation}
We assume that $f_0$ is strictly increasing in this paper. Note that if $f_i(x)=x$ for $i\in\hN$, then $C$ is power loss in the network.

\subsection{The OPF problem}
The OPF problem seeks to minimize the generation cost \eqref{generation cost}, subject to power flow constraints \eqref{PF}, power injection constraints \eqref{constraint s}, and voltage regulation constraints \eqref{constraint v}.
    \begin{align*}
    \textbf{OPF:}~\min~~ & \sum_{i\in\hN} f_i(\re(s_i))\\
	\mathrm{over}~~ & s,V,s_0\\
	\mathrm{s.t.}~~ & s_i = V_i\sum_{j:\,j\sim i}(V_i^*-V_j^*)y_{ij}^*, \quad i\in \hN;\\
	& s_i \in \mathcal{S}_i, \quad i\in \hN^+;\\
	& \underline{V}_i \leq |V_i| \leq \overline{V}_i, \quad i\in \hN^+.
    \end{align*}
	
The challenge in solving the OPF problem comes from the nonconvex quadratic equality constraints in \eqref{PF}. To overcome this challenge, we enlarge the feasible set of OPF to a convex set. To state the convex relaxation, define
	\begin{equation}\label{phi}
	W_{ij} := V_iV_j^*, \qquad i\sim j \text{ or }i=j
	\end{equation}
and let $W:=(W_{ij},i\sim j \text{ or }i=j)$ denote the collection of all such $W_{ij}$. Define
	\begin{equation*}
	W\{i,j\} := \begin{pmatrix}
	W_{ii} & W_{ij} \\
	W_{ji} & W_{jj}
	\end{pmatrix}, \qquad i\sim j,
	\end{equation*}	
then the OPF problem can be equivalently formulated as
\begin{subequations}
\begin{align}
\textbf{OPF':}~\min~~ & \sum_{i\in\hN} f_i(\re(s_i))\nonumber\\
\mathrm{over}~~ & s,W,s_0\nonumber\\
\mathrm{s.t.}~~ & s_i = \sum_{j:\,j\sim i}(W_{ii}-W_{ij})y_{ij}^*, \quad i\in \hN;\label{BIM 1}\\
& s_i \in \mathcal{S}_i, \quad  i\in \hN^+;\label{BIM 2}\\
& \underline{V}_i^2 \leq W_{ii} \leq \overline{V}_i^2, \quad i\in \hN^+;\label{BIM 3}\\
& \rank (W\{i,j\})=1, \quad i\rightarrow j\label{rank}
\end{align}
\end{subequations}
for tree networks according to Theorem \ref{thm: equivalence}, which is proved in Appendix \ref{app: equivalence}. Theorem \ref{thm: equivalence} establishes a bijective map between the feasible set $\hF_\text{OPF}$ of the OPF problem and the feasible set $\hF_\text{OPF'}$ of the OPF' problem, that preserves the objective value. To state the theorem, for any feasible point $x=(s,V,s_0)$ of the OPF problem, define a map $\phi(x):=(s,W,s_0)$ where $W$ is defined according to \eqref{phi}.
	\begin{theorem}\label{thm: equivalence}
	 For any $x=(s,V,s_0)\in\hF_\text{OPF}$, the point $\phi(x)=(s,W,s_0)\in\hF_\text{OPF'}$. Furthermore, the map $\phi:\hF_\text{OPF}\rightarrow\hF_\text{OPF'}$ is bijective for tree networks.
	\end{theorem}

After transforming OPF to OPF', one can relax OPF' to a convex problem, by relaxing the rank constraints in \eqref{rank} to
	\begin{equation}\label{sdp}
	W\{i,j\}\succeq0, \qquad i\rightarrow j,
	\end{equation}
i.e., matrices $W\{i,j\}$ being positive semidefinite. This leads to a second-order cone programming (SOCP) relaxation \cite{sojoudi2012physics}.
    \begin{align*}
    \textbf{SOCP:}~\min~~ & \sum_{i\in\hN} f_i(\re(s_i)) \\
	\mathrm{over}~~ & s,W,s_0\\
	\mathrm{s.t.}~~ & \eqref{BIM 1}-\eqref{BIM 3};\\
	& W\{i,j\}\succeq0, \qquad i\rightarrow j.
    \end{align*}

If the solution $w=(s,W,s_0)$ of the SOCP relaxation satisfies \eqref{rank}, then $w$ is a global optimum of the OPF' problem. This motivates a definition of {\it exactness} as follows.
	\begin{definition}\label{def: exact}
	The SOCP relaxation is {\it exact} if every of its solutions satisfies \eqref{rank}.
	\end{definition}
If the SOCP relaxation is exact, then a global optimum of the OPF problem can be recovered.

\subsection{Related work}\label{sec: related work}
This paper studies the exactness of the SOCP relaxation. Before this paper, several conditions have been derived that guarantee the exactness of the SOCP relaxation \cite{Masoud11,Zhang12,Bose12,Gan12,lam2012distributed,sojoudi2012physics,Javad12}. 

It is proved in \cite{Masoud11} that the SOCP relaxation is exact if there are no lower bounds on the power injections. The results in \cite{Zhang12,Bose12} generalizes this condition.
	\begin{proposition}[\cite{Masoud11}]\label{prop: load over satisfaction}
	If $f_i$ is strictly increasing for $i\in\hN$, and there exists $\overline{p}_i$ and $\overline{q}_i$ such that
	\begin{equation}\label{restrictive S constraints}
	\mathcal{S}_i=\{s\in\mathbb{C} ~|~ \re(s)\leq\overline{p}_i,~\im(s)\leq\overline{q}_i\}
	\end{equation}
	for $i\in\hN^+$, then the SOCP relaxation is exact.
	\end{proposition}
In practice, the constraint sets $\mathcal{S}_i$ take other forms than those required in Proposition \ref{prop: load over satisfaction} \cite{farivar2011optimal}.

In contrast, the conditions in \cite{Gan12} relax the restrictions on $\mathcal{S}_i$, but introduce restrictions on the voltage constraints. To state the result, for every $i\rightarrow j$, let
	\begin{eqnarray}
    \hat{S}_{ij}(p+\ii q ) &\eqdef& \hat{P}_{ij}(p)+\ii \hat{Q}_{ij}(q ) \nonumber\\
    &\eqdef& \sum_{k:\,i\in \hP_k} p_k +\ii \sum_{k:\,i\in \hP_k} q_k \label{lin S}
    \end{eqnarray}
denote the total power injection in the subtree rooted at bus $i$.
	
	\begin{proposition}[\cite{Gan12}]\label{prop: v}
	Assume $f_0$ is strictly increasing and there exists $\overline{p}_i$ and $\overline{q}_i$ such that
	\begin{equation}\label{relax constraints}
	\mathcal{S}_i\subseteq\{s\in\mathbb{C}~|~\re(s)\leq\overline{p}_i,~\im(s)\leq\overline{q}_i\}
	\end{equation}
	for $i\in\hN^+$. Then the SOCP relaxation is exact if $\overline{V}_i=\infty$ for $i\in \hN^+$ and any one of the following conditions hold:
	\begin{itemize}
	\item[(i)] $\hat{P}_{ij}(\overline{p})\leq0$ and $\hat{Q}_{ij}(\overline{q})\leq0$ for all $i\rightarrow j$.
	\item[(ii)] $r_{ij}/x_{ij}=r_{jk}/x_{jk}$ for all $i\rightarrow j$, $j\rightarrow k$.
	\item[(iii)] $r_{ij}/x_{ij}\geq r_{jk}/x_{jk}$ for all $i\rightarrow j$, $j\rightarrow k$, and $\hat{P}_{ij}(\overline{p})\leq0$ for all $i\rightarrow j$.
	\item[(iv)] $r_{ij}/x_{ij}\leq r_{jk}/x_{jk}$ for all $i\rightarrow j$, $j\rightarrow k$, and $\hat{Q}_{ij}(\overline{q})\leq0$ for all $i\rightarrow j$.
	\end{itemize}
	\end{proposition}
In distribution networks, the constraints $|V_i|\leq\overline{V}_i$ cannot be ignored, especially with distributed generators making the voltages likely to exceed $\overline{V}$.

To summarize, all sufficient conditions in literature that guarantee the exactness of the SOCP relaxation require removing some of the constraints in the OPF problem. In fact, the SOCP relaxation is in general not exact, and a 2-bus example is provided in Appendix \ref{app: non exact}. 
\section{A modified OPF problem}\label{sec: mopf}
We answer the following two questions in this section:
	\begin{itemize}
	\item Under what conditions is the SOCP relaxation exact?
	\item Can we modify the OPF problem to enforce these conditions?
	\end{itemize}
More specifically, we give a condition that ensures the exactness of the SOCP relaxation in Section \ref{sec: exact}, and show how the OPF problem can be modified to satisfy this condition in Section \ref{sec: modification}. It will be shown in Section \ref{sec: closeness} that the modification is ``small'' for three test networks.

\subsection{A sufficient condition}\label{sec: exact}
A sufficient condition that guarantees the SOCP relaxation being exact is provided in this section. The condition builds on a linear approximation of the power flow in ``the worst case''.

To state the condition, we first define the linear approximation. Define
	\begin{equation*}
	\hat{W}_{ii}(s) \eqdef W_{00} + 2\sum_{(j,k)\in \hP_i}\re\left(z_{jk}^*\hat{S}_{jk}( s)\right)
	\end{equation*}
for $i\in \hN$ and $s\in\mathbb{C}^{n}$, then $\hat{W}_{ii}(s)$ is a linear approximation of $W_{ii}=|V_i|^2$ (linear in $s$). Define
	\begin{equation}\label{PF S}
	S_{ij}:=P_{ij}+\ii Q_{ij} := (W_{ii}-W_{ij}) y_{ij}^*
	\end{equation}
as the sending-end power flow from bus $i$ to bus $j$ for $i\rightarrow j$, then $\hat{S}_{ij}(s)$ (defined in \eqref{lin S}) is a linear approximation of $S_{ij}$ (linear in $s$).

The linear approximations $\hat{W}_{ii}(s)$ and $\hat{S}_{ij}(s)$ are upper bounds on $W_{ii}$ and $S_{ij}$, as stated in Lemma \ref{lemma: v}, which is proved in Appendix \ref{app: lemma v}. To state the lemma, let $S:=(S_{ij},i\rightarrow j)$ denote the collection of power flow on all lines. For two complex numbers $a,b\in\mathbb{C}$, define the operator $\leq$ by
	$$a\leq b ~\overset{\text{def}}\Longleftrightarrow~ \re(a)\leq \re(b) \text{ and }\im(a)\leq \im(b).$$
	\begin{lemma}\label{lemma: v}
	If $(s,S,W,s_0)$ satisfies \eqref{BIM 1}, \eqref{sdp} and \eqref{PF S}, then $S_{ij}\leq \hat{S}_{ij}(s )$ for $i\rightarrow j$ and $W_{ii}\leq \hat{W}_{ii}(s)$ for $i\in \hN$.
	\end{lemma}

The linear approximations $\hat{S}_{ij}(s)$ and $\hat{W}_{ii}(s)$ are close to $S_{ij}$ and $W_{ii}$ in practice. It can be verified that they satisfy
	\begin{eqnarray*}
	\hat{S}_{jk} &=& s_j+\sum_{i:\,i\rightarrow j}\hat{S}_{ij}, \qquad\qquad\! j\rightarrow k;\\
	\hat{W}_{jj} &=& \hat{W}_{ii}-2\re(z_{ij}^*\hat{S}_{ij}), \qquad i\rightarrow j,
	\end{eqnarray*}
which is called {\it Linear DistFlow model} in literature and known to approximate the exact power flow well. In fact, $\hat{S}_{ij}(s)$ and $\hat{W}_{ii}(s)$ have been widely used in literature, e.g., to study the optimal placement and sizing of shunt capacitors \cite{Baran89_capacitor_placement,Baran89_capacitor_sizing}, to minimize power loss and balance load \cite{Baran89_network_reconfiguration}, and to control reactive power injections for voltage regulation \cite{Turitsyn10}.

The sufficient condition we derive for the exactness of the SOCP relaxation is based on the linear approximations $\hat{S}_{ij}(p+\ii q)=\hat{P}_{ij}(p)+\ii \hat{Q}_{ij}(q)$ and $\hat{W}_{ii}(s)$ of the power flow. In particular, assume there exists $\overline{p}_i$ and $\overline{q}_i$ such that \eqref{relax constraints} holds for $i\in\hN^+$, then the condition depends on $\hat{P}_{ij}(\overline{p})$ and $\hat{Q}_{ij}(\overline{q})$, i.e., upper bounds on the power flow.

To state the condition, define $x^+ \eqdef \max\{ x, 0 \}$ for $x\in\mathbb{R}$, let $a_0^1=1$, $a_0^2=0$, $a_0^3=0$, $a_0^4=1$, and define
	\begin{eqnarray*}
	a_i^1 &:=& \prod_{(j,k)\in \hP_i}\left( 1-\frac{2r_{jk}\hat{P}_{jk}^+(\overline{p})}{\underline{V}_j^2} \right),\\
	a_i^2 &:=& \sum_{(j,k)\in \hP_i}\frac{2r_{jk}\hat{Q}_{jk}^+(\overline{q})}{\underline{V}_j^2} ,\\
	a_i^3 &:=& \sum_{(j,k)\in \hP_i}\frac{2x_{jk}\hat{P}_{jk}^+(\overline{p})}{\underline{V}_j^2} ,\\
	a_i^4 &:=& \prod_{(j,k)\in \hP_i}\left( 1-\frac{2x_{jk}\hat{Q}_{jk}^+(\overline{q})}{\underline{V}_j^2} \right)
	\end{eqnarray*}
for $i\in \hN^+$.
    \begin{lemma}\label{lemma: exact}
    Assume $f_0$ is strictly increasing and there exists $\overline{p}_i$ and $\overline{q}_i$ such that \eqref{relax constraints} holds for $i\in\hN^+$. Then the SOCP relaxation is exact if all of its optimal solutions $w=(s, W, s_0)$ satisfy $\hat{W}_{ii}(s)\leq \overline{V}_i^2$ for $i\in \hN^+$ and
    \begin{itemize}
	\item[\bf C1] $a_j^1r_{ij}>a_j^2x_{ij}$, $a_j^3r_{ij}<a_j^4x_{ij}$ for all $i\rightarrow j$.
	\end{itemize}
    \end{lemma}
The lemma is proved in Appendix \ref{sec: proofs}. However, the requirement $\hat{W}_{ii}(s)\leq \overline{V}_i^2$ depends on solutions of the SOCP relaxation, and cannot be checked apriori. This fact motivates us to modify the OPF problem in Section \ref{sec: modification}.

\subsection{A modified OPF problem}\label{sec: modification}
We modify the OPF problem by imposing additional constraints
	\begin{equation}\label{modify}
	\hat{W}_{ii}(s) \leq \overline{V}_i^2, \quad i\in \hN^+,
	\end{equation}
so that the requirement $\hat{W}_{ii}(s)\leq \overline{|V_i|}^2$ in Lemma \ref{lemma: exact} holds automatically. Note that the constraints \eqref{BIM 3} and \eqref{modify} can then be combined as
	$$\underline{V}_i^2 \leq W_{ii}, ~\hat{W}_{ii}(s)\leq\overline{V}_i^2, \qquad i\in N^+$$
since $W_{ii}\leq \hat{W}_{ii}(s)$ (Lemma \ref{lemma: v}).

To summarize, the modified OPF problem is
	\begin{subequations}
    \begin{align}
    \textbf{OPF-m:}~\min~~ & \sum_{i\in\hN}f_i(\re(s_i)) \nonumber\\
	\mathrm{over}~~ & s,W,s_0\nonumber\\
	\mathrm{s.t.}~~ & s_i = \sum_{j:\,j\sim i}(W_{ii}-W_{ij})y_{ij}^*, \quad i\in \hN;\label{MOPF 1}\\
	& s_i \in \mathcal{S}_i, \quad  i\in \hN^+;\label{MOPF 2}\\
	& \underline{V}_i^2 \leq W_{ii},~\hat{W}_{ii}(s)\leq\overline{V}_i^2, \quad i\in N^+;\label{MOPF 3}\\
	& \rank(W\{i,j\})=1, \quad i\rightarrow j.\label{MOPF 4}
    \end{align}
	\end{subequations}
Note that modifying the OPF problem is necessary to obtain an exact SOCP relaxation, since the SOCP relaxation is in general not exact.

Remarkably, the feasible sets of the OPF-m problem and the OPF problem are similar since $\hat{W}_{ii}(s)$ is close to $W_{ii}$, which is justified by the empirical studies in Section \ref{sec: closeness}.

The SOCP relaxation of the modified OPF problem is called SOCP-m and presented below.
    \begin{align*}
    \textbf{SOCP-m:}~\min~~ & \sum_{i\in\hN}f_i(\re(s_i)) \nonumber\\
	\mathrm{over}~~ & s,W,s_0\nonumber\\
	\mathrm{s.t.}~~ & \eqref{MOPF 1}-\eqref{MOPF 3};\\
	& W\{i,j\}\succeq0, \quad i\rightarrow j.
    \end{align*}

The main contribution of this paper is to provide a sufficient condition for the exactness of SOCP-m, that can be checked apriori and holds in practice. In particular, the condition is given in Theorem \ref{thm: exact}, which follows directly from Lemma \ref{lemma: exact}.
	\begin{theorem}\label{thm: exact}
	Assume $f_0$ is strictly increasing and there exists $\overline{p}_i$ and $\overline{q}_i$ such that \eqref{relax constraints} holds for $i\in\hN^+$. Then the SOCP-m relaxation is exact if C1 holds.
	\end{theorem}
	
C1 can be checked apriori since it does not depend on the solutions of the SOCP-m relaxation. In fact, $\{a_j^k\}_{j\in \hN, k=1,2,3,4}$ are functions of $(r,x,\overline{p},\overline{q},\underline{V})$ that can be computed in $O(n)$ time, therefore the complexity of checking C1 is $O(n)$.

C1 requires $\overline{p}$ and $\overline{q}$ to be ``small''. Fix $(r,x,\underline{V})$, then C1 is a condition on $(\overline{p},\overline{q})$. It can be verified that if $(\overline{p},\overline{q}) \leq (\overline{p}',\overline{q}')$ componentwise, then
    $$\text{C1 holds for }(\overline{p}',\overline{q}') ~\Rightarrow~ \text{C1 holds for }(\overline{p},\overline{q}),$$
i.e., the smaller the power injections, the more likely C1 holds. In particular, it can be verified that if $(\overline{p},\overline{q})\leq(0,0)$, i.e., there is no distributed generation, then C1 holds as long as $(r,x)>0$ componentwise.

As will be seen in the empirical studies in Section \ref{sec: widely hold}, C1 holds for three test networks, including those with high penetration of distributed generation, i.e., big $(\overline{p},\overline{q})$.


\subsection{Uniqueness of solutions}
If the solution of the SOCP-m relaxation is unique, then any convex programming solver will obtain the same solution.
	\begin{theorem}\label{thm: unique}
	If $f_i$ is convex for $i\in \hN$; $\mathcal{S}_i$ is convex for $i\in\hN^+$; and the SOCP-m relaxation is exact, then the SOCP-m relaxation has at most one solution.
	\end{theorem}
The theorem is proved in Appendix \ref{app: thm unique}.
	
\section{Case Studies}\label{sec: case study}
In this section we use three test networks to demonstrate the following two arguments made in Section \ref{sec: mopf}:
\begin{enumerate}
\item the feasible sets of the OPF problem and the OPF-m problem are close;
\item Condition C1 holds.
\end{enumerate}

\subsection{Test networks}
\label{sec: test networks}
We consider three test networks: the IEEE 13-bus test network \cite{IEEE} and two real-world networks in the service territory of Southern California Edison (SCE), a utility company in Southern California \cite{SCE}.

The IEEE 13-bus test network is an unbalanced three-phase network with unmodeled devices including regulators, circuit switches, and split transformers. It is adjusted as follows to be modeled by the power flow model in \eqref{PF}.
\begin{enumerate}
\item Assume that each bus has three phases and split the load uniformly among the three phases.
\item Assume that the three phases are decoupled so that the network becomes three identical single phase networks.
\item Assume that circuit switches are in their normal operating states, think of regulators as substations since they have fixed voltages, ignore split transformers and place the corresponding load at the primary side.
\end{enumerate}

We also consider two real-world networks, a 47-bus network and a 56-bus network, in the service territory of SCE. Both networks have high penetration of distributed generation. Their topologies are shown in Fig. \ref{fig:circuit}, and their parameters are summarized in Table \ref{table:data} and \ref{table: 56} respectively.
	\begin{figure*}[!htbp]
    \centering
    \includegraphics[scale=0.42]{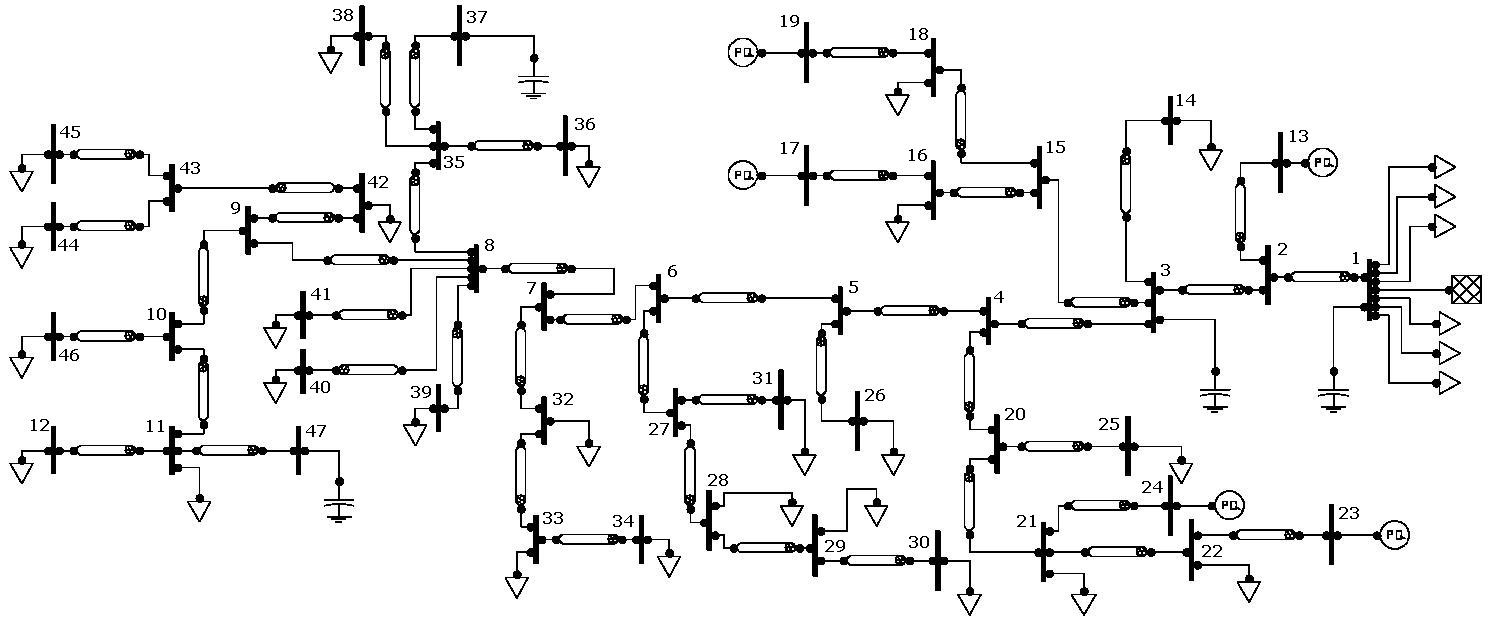}
    \includegraphics[scale=0.42]{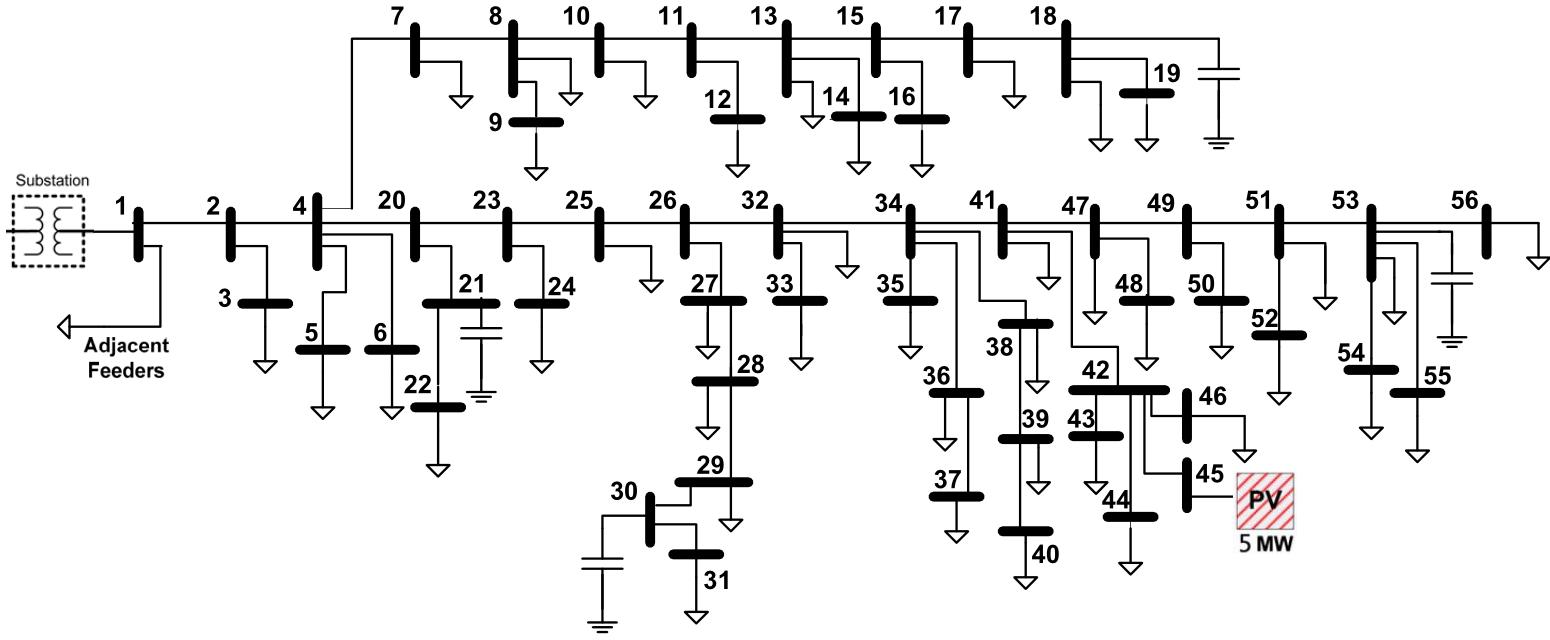}
    \caption{Topologies of the 47-bus and 56-bus SCE test networks \cite{Masoud11,farivar2011optimal}.}
    \label{fig:circuit}
	\end{figure*}
	

\begin{table*}
\caption{Line impedances, peak spot load, and nameplate ratings of capacitors and PV generators of the 47-bus network.}
\centering
\scriptsize
\begin{tabular}{|p{0.42cm}|p{0.3cm}|m{0.43cm}|b{0.43cm}   |p{0.42cm}|p{0.3cm}|p{0.43cm}|p{0.43cm}    |p{0.42cm}|p{0.3cm}|p{0.43cm}|p{0.43cm} |p{0.3cm}|p{0.35cm}|c|c|c|c|}
\hline
\multicolumn{18}{|c|}{Network Data}\\
\hline
\multicolumn{4}{|c|}{Line Data}& \multicolumn{4}{|c|}{Line Data}& \multicolumn{4}{|c|}{Line Data}& \multicolumn{2}{|c|}{Load Data}& \multicolumn{2}{|c|}{Load Data}&\multicolumn{2}{|c|}{PV Generators}\\
\hline
From & To & R & X & From & To & R & X & From & To & R & X& Bus& Peak & Bus& Peak & Bus&Nameplate \\
Bus & Bus &$(\Omega)$& $(\Omega)$ & Bus & Bus & $(\Omega)$ & $(\Omega)$ & Bus & Bus & $(\Omega)$ & $(\Omega)$ & No &  MVA& No & MVA& No & Capacity\\
\hline
1	&	2	&	0.259	&	0.808	&	8	&	41	&	0.107	 &	0.031	&	 21	&	22	 &	 0.198	 &	 0.046	 &	1	&	30	&	34	&	0.2	&				 &		 \\	
2	&	13	&	0	&	0	&	8	&	35	&	0.076	 &	0.015	&	22	&	 23	&	0	 &	 0	 &	 11	&	 0.67	 &	36	&	0.27	&			 13	&	 1.5MW	\\	
2	&	3	&	0.031	&	0.092	&	8	&	9	&	0.031	 &	0.031	&	 27	&	31	 &	 0.046	 &	 0.015	 &	12	&	0.45	 &	38	&	0.45	 &			 17	&	 0.4MW	 \\	
3	&	4	&	0.046	&	0.092	&	9	&	10	&	0.015	 &	0.015	&	 27	&	28	 &	 0.107	 &	 0.031	 &	14	&	0.89	 &	39	&	1.34	 &			 19	&	 1.5 MW	 \\	
3	&	14	&	0.092	&	0.031	&	9	&	42	&	0.153	 &	0.046	&	 28	&	29	 &	 0.107	 &	 0.031	 &	16	&	0.07	 &	40	&	0.13	 &			 23	&	 1 MW	 \\	
3	&	15	&	0.214	&	0.046	&	10	&	11	&	0.107	 &	0.076	&	 29	&	30	 &	 0.061	 &	 0.015	 &	18	&	0.67	 &	41	&	0.67	 &			 24	&	 2 MW	 \\	
4	&	20	&	0.336	&	0.061	&	10	&	46	&	0.229	 &	0.122	&	 32	&	33	 &	 0.046	 &	 0.015	 &	21	&	0.45	 &	42	&	0.13	 &			 &		 \\\cline{17-18}		
4	&	5	&	0.107	&	0.183	&	11	&	47	&	0.031	 &	0.015	&	 33	&	34	 &	 0.031	 &	0.010	 &	22	&	2.23	 &	44	&	0.45	&			 \multicolumn{2}{c|}{Shunt Capacitors} 			 \\  \cline{17-18}	
5	&	26	&	0.061	&	0.015	&	11	&	12	&	0.076	 &	0.046	&	 35	&	36	 &	 0.076	 &	 0.015	 &	25	&	0.45	 &	45	&	0.2	&			 \multicolumn{1}{c|}{Bus} &	 \multicolumn{1}{c|}{Nameplate}	\\		
5	&	6	&	0.015	&	0.031	&	15	&	18	&	0.046	 &	0.015	&	 35	&	37	 &	 0.076	 &	 0.046	 &	26	&	0.2	 &	46	&	0.45	&			 \multicolumn{1}{c|}{No.} &	 \multicolumn{1}{c|}{Capacity}	\\ \cline{15-18}		 
6	&	27	&	0.168	&	0.061	&	15	&	16	&	0.107	 &	0.015	&	 35	&	38	 &	 0.107	 &	 0.015	 &	28	&	0.13	 &	 \multicolumn{2}{c|}{ } 		&		&		 \\				
6	&	7	&	0.031	&	0.046	&	16	&	17	&	0	 &	0	&	42	&	 43	&	 0.061	 &	 0.015	&	 29	&	0.13	 &	 \multicolumn{2}{c|}{Base Voltage (kV) = 12.35}		 &	1	 &	 6000 kVAR	\\				
7	&	32	&	0.076	&	0.015	&	18	&	19	&	0	 &	0	&	43	&	 44	&	 0.061	 &	 0.015	&	 30	&	0.2	 &	\multicolumn{2}{c|}{Base kVA = 1000} 	 &	3	&	 1200 kVAR	 \\					
7	&	8	&	0.015	&	0.015	&	20	&	21	&	0.122	 &	0.092	&	 43	&	45	 &	 0.061	 &	 0.015	 &	31	&	0.07	 &	 \multicolumn{2}{c|}{Substation Voltage = 12.35}    &	 37 &		 1800 kVAR	 \\						
8	&	40	&	0.046	&	0.015	&	20	&	25	&	0.214	 &	0.046	&		 &		 &		 &		 &	 32	&	0.13	 &	\multicolumn{2}{c|}{}	 &	47		 & 1800 kVAR  \\						
8	&	39	&	0.244	&	0.046	&	21	&	24	&	0	 &	0	&		&		 &		 &		 &	 33	&	 0.27	 &	\multicolumn{2}{c|}{}	& & \\
\hline
\end{tabular}
\label{table:data}
\end{table*}

\begin{centering}
\begin{table*}

\caption{Line impedances, peak spot load, and nameplate ratings of capacitors and PV generators
of the 56-bus network.}
\centering
\scriptsize
\begin{tabular}{|c|c|c|c|c|c|c|c|c|c|c|c|c|c|c|c|c|c|}
\hline
\multicolumn{18}{|c|}{Network Data}\\
\hline
\multicolumn{4}{|c|}{Line Data}& \multicolumn{4}{|c|}{Line Data}& \multicolumn{4}{|c|}{Line Data}& \multicolumn{2}{|c|}{Load Data}& \multicolumn{2}{|c|}{Load Data}&\multicolumn{2}{|c|}{Load Data}\\
\hline
From&To&R&X&From&To& R& X& From& To& R& X& Bus& Peak & Bus& Peak &  Bus &Peak  \\
Bus.&Bus.&$(\Omega)$& $(\Omega)$ & Bus. & Bus. & $(\Omega)$ & $(\Omega)$ & Bus.& Bus.& $(\Omega)$ & $(\Omega)$ & No.&  MVA& No.& MVA& No.& MVA\\
\hline

1	&	2	&	0.160	&	0.388	&	20	&	21	&	0.251	&	0.096	&	 39	&	40	 &	 2.349	 &	 0.964	&	3	&	0.057	&	29  &	0.044  & 52& 0.315	 \\
2	&	3	&	0.824	&	0.315	&	21	&	22	&	1.818	&	0.695	&	 34	&	41	 &	 0.115	 &	 0.278	&	5	&	0.121	&	31	&	0.053  & 54& 	 0.061	 \\
2	&	4	&	0.144	&	0.349	&	20	&	23	&	0.225	&	0.542	&	 41	&	42	 &	 0.159	 &	 0.384	&	6	&	0.049	&	32	&	0.223 & 55&	 0.055	 \\
4	&	5	&	1.026	&	0.421	&	23	&	24	&	0.127	&	0.028	&	 42	&	43	 &	 0.934	 &	 0.383	&	7	&	0.053	&	33	&	0.123 & 56&	 0.130	 \\\cline{17-18}
4	&	6	&	0.741	&	0.466   &	23	&	25	&	0.284	&	0.687	&	 42	&	44	 &	 0.506	 &	 0.163	&	8	&	0.047	&	34	&	0.067 & \multicolumn{2}{c|}{Shunt Cap}	 \\\cline{17-18}
4	&	7	&	0.528	&	0.468	&	25	&	26	&	0.171	&	0.414	&	 42	&	45	 &	 0.095	 &	 0.195	&	9	&	0.068	&	35	&	0.094&   \multicolumn{1}{c|}{Bus} &	 \multicolumn{1}{c|}{Mvar}			 \\\cline{17-18}
7	&	8	&	0.358	&	0.314	&	26	&	27	&	0.414	&	0.386	&	 42	&	46	 &	 1.915	 &	 0.769	&	10	&	0.048	&	36	&	0.097&  19& 	 0.6 	 \\
8	&	9	&	2.032	&	0.798	&	27	&	28	&	0.210	&	0.196	&	 41	&	47	 &	 0.157	 &	 0.379	&	11	&	0.067	&	37	&	0.281&  21&	 0.6 	 \\
8	&	10	&	0.502	&	0.441	&	28	&	29	&	0.395	&	0.369	&	 47	&	48	 &	 1.641	 &	 0.670	&	12	&	0.094	&	38	&	0.117&  30&	 0.6 		 \\
10	&	11	&	0.372	&	0.327	&	29	&	30	&	0.248	&	0.232	&	 47	&	49	 &	 0.081	 &	 0.196	&	14	&	0.057	&	39	&	0.131& 53&	 0.6 		 \\\cline{17-18}
11	&	12	&	1.431	&	0.999	&	30	&	31	&	0.279	&	0.260	&	 49	&	50	 &	 1.727	 &	 0.709	&	16	&	0.053	&	40	&	0.030& \multicolumn{2}{c|}{Photovoltaic}		 \\\cline{17-18}
11	&	13	&	0.429	&	0.377	&	26	&	32	&	0.205	&	0.495	&	 49	&	51	 &	 0.112	 &	 0.270	&	17	&	0.057	&	41	&	0.046& \multicolumn{1}{c|}{Bus} &	 \multicolumn{1}{c|}{Capacity}		 \\\cline{17-18}
13	&	14	&	0.671	&	0.257	&	32	&	33	&	0.263	&	0.073	&	 51	&	52	 &	 0.674	 &	 0.275	&	18	&	0.112	&	42	&	0.054&   & \\
13	&	15	&	0.457	&	0.401	&	32	&	34	&	0.071	&	0.171	&	 51	&	53	 &	 0.070	 &	 0.170	&	19	&	0.087	&	43	&	0.083&   45 &		 5MW		 \\\cline{17-18}
15	&	16	&	1.008	&	0.385	&	34	&	35	&	0.625	&	0.273	&	 53	&	54	 &	 2.041	 &	 0.780	&	22	&	0.063	&	44	&	0.057&  \multicolumn{2}{c|}{} \\
15	&	17	&	0.153	&	0.134	&	34	&	36	&	0.510	&	0.209	&	 53	&	55	 &	 0.813	 &	 0.334	&	24	&	0.135	&	46	&	0.134&  \multicolumn{2}{c|}{$V_\textrm{base}$ = 12kV}	 \\
17	&	18	&	0.971	&	0.722	&	36	&	37	&	2.018	&	0.829	&	 53	&	56	 &	 0.141	 &	 0.340	&	25	&	0.100	&	47	&	0.045& \multicolumn{2}{c|}{$S_\textrm{base}$ = 1MVA} 	 \\
18	&	19	&	1.885	&	0.721	&	34	&	38	&	1.062	&	0.406	&		 &		 &		 &		 &	27	&	0.048	&	48	&	0.196&  \multicolumn{2}{c|}{$Z_\textrm{base}= 144 \Omega$ }		 \\
4	&	20	&	0.138	&	0.334	&	38	&	39	&	0.610	&	0.238	&		 &		 &		 &		 &	28	&	0.038	&	50	&	0.045 &  \multicolumn{2}{c|}{}		 \\
		
\hline

\end{tabular}
\label{table: 56}
\end{table*}
\end{centering}

The three networks have increasing penetration of distributed generation. While the IEEE 13-bus network does not have any distributed generation (0\% penetration), the SCE 47-bus network has $6.4$MW nameplate distributed generation capacity (over 50\% penetration in comparison with $11.3$MVA peak spot load) \cite{Masoud11}, and the SCE 56-bus network has $5$MW nameplate distributed generation capacity (over 100\% penetration in comparison with $3.835$MVA peak spot load) \cite{farivar2011optimal}.

\subsection{Feasible sets of OPF-m and OPF' are similar}
\label{sec: closeness}
We show that the feasible sets of the OPF-m problem and the OPF' problem are similar for all three test networks in this section. More specifically, we show that the OPF-m problem eliminates some feasible points of the OPF' problem that are close to the voltage upper bounds.

To state the results, we define a measure that will be used to evaluate the difference between the feasible sets of the OPF' problem and the OPF-m problem. It is claimed in \cite{chiang1990existence} that given $V_0$ and $s$, there exists a unique voltage $V(s)$ near the nominal value that satisfies the power flow \eqref{PF}. Define
	\begin{equation*}
	\varepsilon := \max\left\{\hat{W}_{ii}(s)-|V_i(s)|^2~|~
	s \text{ satisfies }\eqref{constraint s},~i\in \hN\right\}
	\end{equation*}
as the maximum deviation of $\hat{W}_{ii}(s)$ from $W_{ii}(s)=|V_i(s)|^2$. It follows from Lemma \ref{lemma: v} that $\hat{W}_{ii}(s)\geq W_{ii}(s)$ for all $s$ and all $i\in\hN$, therefore $\varepsilon\geq0$.

The value ``$\varepsilon$'' serves as a measure of the difference between the feasible sets of the OPF-m problem and the OPF' problem for the following reason. Consider the OPF' problem with stricter voltage upper bound constraints $W_{ii}\leq\overline{V}_i^2-\varepsilon$:
	\begin{eqnarray*}
	\textbf{OPF'-$\varepsilon$: }\min && \sum_{i\in\hN}f_i(\re(s_i)) \\
	\mathrm{over}  & & s,W,s_0\nonumber\\
	\mathrm{s.t.} && \eqref{BIM 1},\eqref{BIM 2}, \eqref{rank};\\
	&& \underline{V}_i^2 \leq W_{ii} \leq \overline{V}_i^2-\varepsilon, \quad i\in \hN^+.
	\end{eqnarray*}
Then it follows from
	\begin{equation*}
	W_{ii}(s)\leq\overline{V}_i^2-\varepsilon ~~ \Longrightarrow ~~ \hat{W}_{ii}(s)\leq\overline{V}_i^2,
	\qquad i\in \hN^+
	\end{equation*}
that the feasible set $\hF_{\text{OPF-'}\varepsilon}$ of OPF'-$\varepsilon$ is contained in the feasible set $\hF_\text{OPF-m}$ of OPF-m. Furthermore, we know that the feasible set of OPF-m is contained in the feasible set of OPF'. Hence, $$\hF_{\text{OPF'-}\varepsilon}\subseteq\hF_\text{OPF-m}\subseteq\hF_\text{OPF'}.$$

To summarize, OPF-m is ``sandwiched'' between OPF' and OPF'-$\varepsilon$ as illustrated in Fig. \ref{fig: MOPF}. If $\varepsilon$ is small, then $\hF_\text{OPF-m}$ is similar to $\hF_\text{OPF'}$.
	\begin{figure}[!htbp]
     	\centering
     	\includegraphics[scale=0.5]{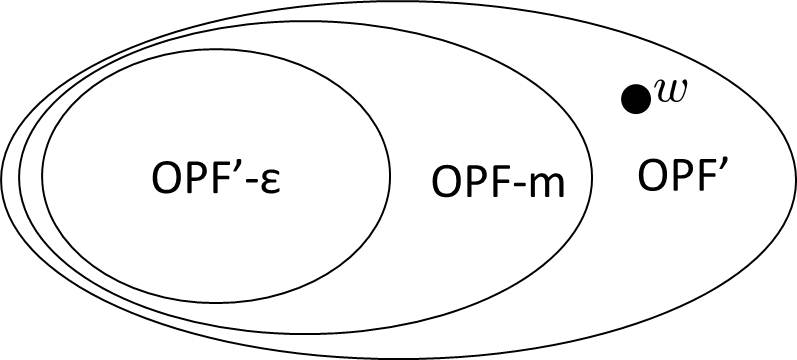}
      	\caption{Feasible sets of OPF'-$\varepsilon$, OPF-m, and OPF'. The point $w$ is feasible for OPF' but not for OPF-m. }
      	\label{fig: MOPF}
	\end{figure}

Moreover, if $\varepsilon$ is small, then any point $w$ that is feasible for OPF' but infeasible for OPF-m is close to the voltage upper bound since $W_{ii}>\overline{V}_i^2-\varepsilon$ for some $i\in \hN^+$. Such points are perhaps undesirable for robust operation.

Now we show that $\varepsilon$ is small for all three test networks. In the numerical studies, we assume that the substation voltage is fixed at the nominal value, i.e., $W_{00}=1$, and that the voltage upper and lower bounds are $\overline{V}_i=1.05$ and $\underline{V}_i=0.95$ for $i\in \hN^+$.

To evaluate $\varepsilon$ for the IEEE 13-bus network, we further assume that $\overline{p}=\underline{p}$, $\overline{q}=\underline{q}$, and that they equal the values specified in the IEEE test case documents. In this setup, $\varepsilon=0.0043$. Therefore the voltage constraints are $0.9025\leq W_{ii}\leq 1.1025$ for OPF' and $0.9025\leq W_{ii}\leq 1.0982$ for OPF'-$\varepsilon$.
    \begin{table}[!h]
	\caption{Closeness of OPF-m and OPF}
	\label{table: epsilon}
	\begin{center}
        \begin{tabular}{| c | c |}
        \hline
        & $\varepsilon$ \\
        \hline
        IEEE 13-bus & 0.0043 \\
        SCE 47-bus & 0.0031 \\
        SCE 56-bus & 0.0106 \\
        \hline
        \end{tabular}
        \end{center}
        \end{table}

To evaluate $\varepsilon$ for the SCE test networks, we further assume that all loads draw peak spot apparent power at power factor 0.97, that all shunt capacitors are switched on, and that distributed generators generate real power at their nameplate capacities with zero reactive power (these assumptions enforce $\underline{p}=\overline{p}$, $\underline{q}=\overline{q}$, and simplify the calculation of $\varepsilon$). The values of $\varepsilon$ are summarized in Table \ref{table: epsilon}. For example, $\varepsilon=0.0031$ for the SCE 47-bus network. Therefore the voltage constraints are
$0.9025\leq W_{ii}\leq 1.1025$ for OPF' and $0.9025\leq W_{ii}\leq 1.0994$ for OPF'-$\varepsilon$.

\subsection{Condition C1 holds}
\label{sec: widely hold}
We show that Condition C1 holds for all three test networks in this section.

To present the results, we transform C1 to a compact form. To state the compact form, first define\footnote{If $a^3_i=0$ for some $i\in \hN$, then set $\overline{b}_i=\infty$. In practice, $a^1_i\approx1$.}
	\begin{equation*}
	\underline{b}_i :=  \frac{a^2_i}{a^1_i},~\overline{b}_i :=  \frac{a^4_i}{a^3_i}
	\end{equation*}
for $i\in \hN$, then C1 is equivalent to
	\begin{equation}\label{condition}
	\frac{r_{ij}}{x_{ij}}  \in\left(\underline{b}_j,\overline{b}_j\right),
		\qquad i\rightarrow j.
	\end{equation}
We have checked that \eqref{condition} holds for all three test networks.

To better present the result, note that \eqref{condition} is implied by
    \begin{equation}\label{display}
	\conv\left(\left\{\frac{r_{ij}}{x_{ij}},i\rightarrow j\right\}\right) \quad \subseteq \quad \underset{j\in \hN}{\cap}\left(\underline{b}_j,\overline{b}_j\right)
	\end{equation}
where $\conv(A)$ denotes the convex hull of a set $A$. In the rest of this section, we focus on \eqref{display} since it only involves 2 intervals and is therefore easier to present.

We call the left hand side of \eqref{display} the {\it range of $r/x$} and the right hand side the {\it minimum interval}. The calculation of ranges of $r/x$ is straightforward. To calculate the minimum intervals of the test networks, we consider two cases: a bad case and the worst case. In the bad case, we set the bounds $\overline{p}$ and $\overline{q}$ as follows:
\begin{itemize}
\item For a load bus $i$, we set $(\overline{p}_i,\overline{q}_i)$ to equal to the specified load data\footnote{In the SCE networks, only apparent power is given. Therefore we assume a power factor of 0.97 to obtain the real and reactive power consumptions. For example, we set $\overline{p}_{22}=-2.16$MW and $\overline{q}_{22}=-0.54$MVAR at load bus 22 since it draws 2.23MVA apparent power.} because there is usually not much flexibility in controlling the loads.
\item For a shunt capacitor bus $i$, we set $\overline{p}_i=0$ and $\overline{q}_i$ to equal to its nameplate capacity.
\item For a distributed generator bus $i$, we set $\overline{q}_i=0$ and $\overline{p}_i$ to equal to its nameplate capacity. In practice, $\overline{p}_i$ is usually smaller.
\end{itemize}
In the bad case setup, $(\overline{p}_i,\overline{q}_i)$ is artificially enlarged except for load buses.

In the worst case, we further set $\overline{p}_i=0$ and $\overline{q}_i=0$ for load buses while they are negative in practice. Hence in the worst case setup, $(\overline{p}_i,\overline{q}_i)$ is artificially enlarged for all buses.

The minimum intervals of the three test networks in the two cases are summarized in Table \ref{table: C1}. Noting that C1 is more difficult to hold if $(\overline{p},\overline{q})$ gets bigger, and the test networks have increasing penetration of distributed generation, one expects C1 to be less likely to hold in the test networks.

	\begin{table}[!h]
	\caption{The range of $r/x$ and minimum intervals of test networks}
	\label{table: C1}
	\begin{center}
        \begin{tabular}{| c | c | c | c |}
        \hline
        & \multirow{2}{*}{range of $r/x$} & minimum interval & minimum interval \\
        & & (worst case) & (bad case) \\
        \hline
        IEEE 13-bus & $[0.331, 2.62]$ & $(0.0175, \infty)$ & $(0.0013, \infty)$ \\
        SCE 47-bus & $[0.321, 7.13]$ & (0.0374,10.0) & (0.0187,995) \\
        SCE 56-bus & $[0.414, 4.50]$ & (0.0652,2.93) & (0.0528,5.85) \\
        \hline
        \end{tabular}
        \end{center}
        \end{table}

In the bad case, the minimum interval contains the range of $r/x$ for all three networks with significant margins. In the worst case, the minimum interval covers the range of $r/x$ for the first two networks, but not the third one. However, \eqref{condition}, which is equivalent to C1, still holds for the third network.

To summarize, C1 holds for all three test networks, even those with high penetration of distributed generation. 
\section{Conclusion}
We have proved that the SOCP relaxation of the OPF problem is exact under Condition C1, after imposing additional constraints on the power injections. Condition C1 can be checked apriori, and holds for the IEEE 13-bus network and two real-world networks with high penetration of distributed generation. The additional constraints eliminate some feasible points of the OPF problem that are close to the voltage upper bounds, which is justified using the same set of test networks.

There remains many interesting open questions on finding the global optimum of the OPF problem. For example, is there an apriori guarantee that a convex relaxation be exact for mesh networks? Is there an apriori guarantee that a convex relaxation be exact for unbalanced three-phase tree networks? If the SOCP relaxation is not exact, can its solution be used to obtain some feasible solution of the OPF problem?

\bibliographystyle{IEEEtran}
\bibliography{optimal_power_flow}


\appendices
\section{Proof of Lemma \ref{lemma: exact}}\label{sec: proofs}
We prove Lemma \ref{lemma: exact} in this appendix. The idea of the proof is as follows. Assume the conditions in Lemma \ref{lemma: exact} holds. If there exists an optimal solution $w$ of the SOCP relaxation that is infeasible for OPF', then one can construct another feasible point $w'$ of the SOCP relaxation that has a smaller objective value than $w$, which contradicts with $w$ being optimal. It follows that every optimal solution of the SOCP relaxation is feasible for OPF', i.e., the SOCP relaxation is exact.

The rest of the appendix is structured as follows. The OPF' problem and the SOCP relaxation are transformed to forms that are easier to illustrate the construction of $w'$ in Appendix \ref{app: transformation}. Since notations in the proof are complicated and inhibit comprehension for general tree networks, we first present the proof for one-line networks (in Appendix \ref{app: line}) where the notations can be significantly simplified (as in Appendix \ref{app: notation}), and then present the proof for general tree networks in Appendix \ref{app: tree}.

\subsection{Transformation}\label{app: transformation}
We transform the OPF' problem and the SOCP relaxation to forms that are easier to illustrate the construction of $w'$ in this appendix. More specifically, the forms introduced in \cite{Masoud11}.

Recall the definition of $S_{ij}$ in \eqref{PF S}, define $$v_i:=W_{ii}$$ for $i\in \hN$, and define
	\begin{eqnarray*}
	\ell_{ij}:=|y_{ij}|^2(W_{ii}-W_{ij}-W_{ji}+W_{jj})
	\end{eqnarray*}
for $i\rightarrow j$, then it can be verified that
	$$v_i-v_j = 2\re(z_{ij}^*S_{ij})-|z_{ij}|^2\ell_{ij}$$
for $i\rightarrow j$. 
The power flow constraints \eqref{BIM 1} can be transformed as
	\begin{eqnarray*}
	&& s_i = \sum_{j:\,j\sim i}(W_{ii}-W_{ij})y_{ij}^*\\
	&\Leftrightarrow& s_i = \sum_{j:\,i\rightarrow j}(W_{ii}-W_{ij})y_{ij}^* + \sum_{h:\,h\rightarrow i}(W_{ii}-W_{ih})y_{ih}^*\\
	&\Leftrightarrow& s_i = \sum_{j:\,i\rightarrow j}S_{ij} + \sum_{h:\,h\rightarrow i}(|z_{hi}|^2\ell_{hi}-(W_{hh}-W_{hi}))y_{hi}^*\\
	&\Leftrightarrow& s_i = \sum_{j:\,i\rightarrow j}S_{ij} + \sum_{h:\,h\rightarrow i}(z_{hi}\ell_{hi}-S_{hi})\\
	&\Leftrightarrow& \sum_{h:\,h\rightarrow i}(S_{hi}-z_{hi}\ell_{hi}) + s_i = \sum_{j:\,i\rightarrow j}S_{ij}
	\end{eqnarray*}
for $i\in \hN$, the voltage constraints can be transformed as
	\begin{eqnarray*}
	\underline{V}_i^2 \leq W_{ii} \leq \overline{V}_i^2 ~\Leftrightarrow~ \underline{v}_i \leq v_i \leq \overline{v}_i
	\end{eqnarray*}
where $\underline{v}_i:=\underline{V}_i^2$ and $\overline{v}_i:=\overline{V}_i^2$ for $i\in \hN^+$, the rank constraints can be transformed as
	\begin{eqnarray*}
	&& \rank(W\{i,j\})=1\\
	&\Leftrightarrow& W_{jj}=\frac{W_{ij}W_{ji}}{W_{ii}}\\
	&\Leftrightarrow& W_{jj}-W_{ij}-W_{ji}+W_{ii}=\frac{(W_{ii}-W_{ij})(W_{ii}-W_{ji})}{W_{ii}}\\
	&\Leftrightarrow& \ell_{ij}=\frac{|S_{ij}|^2}{v_i}
	\end{eqnarray*}
for $i\rightarrow j$, and the constraints in \eqref{sdp} can be transformed as
	\begin{eqnarray*}
	W\{i,j\})\succeq0 ~\Leftrightarrow~ \ell_{ij} \geq \frac{|S_{ij}|^2}{v_i}
	\end{eqnarray*}
for $i\rightarrow j$.

To summarize, the OPF' problem can be reformulated as
\begin{subequations}
\begin{align}
\textbf{OPF':}~\min~~ & \sum_{i\in\hN} f_i(\re(s_i))\nonumber\\
\mathrm{over}~~ & s,S,\ell,v,s_0\nonumber\\
\mathrm{s.t.}~~ & v_i-v_j = 2\re(z_{ij}^*S_{ij})-|z_{ij}|^2\ell_{ij},\quad i\rightarrow j;\label{BFM 1}\\
& \sum_{h:\,h\rightarrow i}(S_{hi}-z_{hi}\ell_{hi}) + s_i = \sum_{j:\,i\rightarrow j}S_{ij}, \quad i\in\hN;\label{BFM 2}\\
& \underline{v}_i \leq v_i \leq \overline{v}_i,\quad i\in\hN^+; \label{BFM v}\\
& s_i \in \mathcal{S}_i, \quad i\in \hN^+;\label{BFM 3}\\
& \ell_{ij}=\frac{|S_{ij}|^2}{v_i}, \quad i\rightarrow j,\label{BFM 4}
\end{align}
\end{subequations}
and the SOCP relaxation can be reformulated as
\begin{eqnarray}
\textbf{SOCP:}~\min && \sum_{i\in\hN} f_i(\re(s_i))\nonumber\\
\mathrm{over} && s,S,\ell,v,s_0\nonumber\\
\mathrm{s.t.} && \eqref{BFM 1}-\eqref{BFM 3};\nonumber\\
&& \ell_{ij}\geq\frac{|S_{ij}|^2}{v_i}, \quad i\rightarrow j.\label{relax}
\end{eqnarray}
The SOCP relaxation is exact if and only if every of its solutions satisfies \eqref{BFM 4}.

\subsection{Simplified notations in one-line networks}\label{app: notation}
We focus on one-line networks, where the notations can be significantly simplified for comprehension, in Appendix \ref{app: notation} and \ref{app: line}, to present the key idea of the proof. The proof for tree networks is similar except for complicated notations, which is given in detail in Appendix \ref{app: tree}.

In one-line networks, we can abbreviate $z_{ij}$, $S_{ij}$, and $\ell_{ij}$ by $z_i$, $S_i$, and $\ell_i$ respectively. The notations are summarized in Fig. \ref{figure: one line}.
	\begin{figure}[h]
     	\centering
     	\includegraphics[scale=0.25]{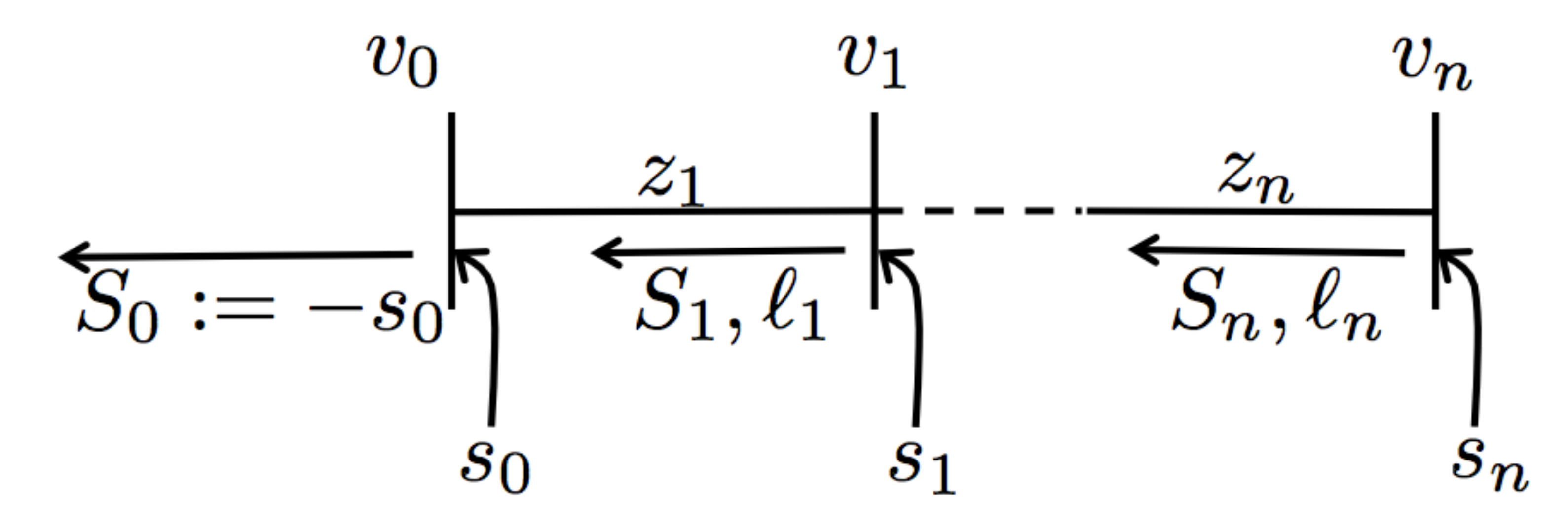}
      	\caption{Simplified notations in a one-line network.}
      	\label{figure: one line}
	\end{figure}
	
With the simplified notations, OPF' is simplified to
	\begin{subequations}
    \begin{align}
    \textbf{OPF':}~\min~~ & \sum_{i=0}^n f_i(\re(s_i))\nonumber\\
	\mathrm{over}~~ & s,S,\ell,v,s_0\nonumber\\
	\mathrm{s.t.}~~ & v_i-v_{i-1} = 2\re(z_i^*S_i) -|z_i|^2\ell_i, \quad i\in\hN^+;\label{Line 2}\\
	& \underline{v}_i \leq v_i \leq \overline{v}_i, \quad i\in\hN^+;\label{Line 3}\\
	& S_n=s_n; \label{Line 6}\\
	& S_{i-1} = s_{i-1} + S_i-z_i\ell_i, \quad i=2,\ldots,n,\nonumber\\
	& 0=s_{i-1} + S_i - z_i\ell_i,\quad i=1; \label{Line 4}\\
	& s_i \in \mathcal{S}_i,   \qquad~ i\in \hN^+;  \label{Line 5}\\
	& \ell_i = \frac{|S_i|^2}{v_i},\quad  i\in \hN^+,\label{Line 7}
    \end{align}
	\end{subequations}
the SOCP relaxation is simplified to
    \begin{align}
    \textbf{SOCP:}~\min~~ & \sum_{i=0}^n f_i(\re(s_i))\nonumber\\
	\mathrm{over}~~ & s,S,\ell,v,s_0\nonumber\\
	\mathrm{s.t.}~~ & \eqref{Line 2}-\eqref{Line 5};\nonumber\\
	& \ell_i \geq \frac{|S_i|^2}{v_i},\quad  i\in \hN^+,\label{line relax}
    \end{align}
the linear approximation $\hat{S}(p+\ii q )$ is simplified to
	$$\hat{S}_i(p+\ii q ) = \hat{P}_i(p ) + \ii \hat{Q}_i(q) = \sum_{j\geq i}p_j + \ii \sum_{j\geq i}q_j, \quad i\in \hN^+,$$
the linear approximation $\hat{W}_{ii}(s)$ is simplified to
	$$\hat{v}_i(s) := \hat{W}_{ii}(s) = v_0 + 2\sum_{j=1}^i\re\left(z_j\hat{S}_j( s)\right), \quad i\in \hN,$$
quantities $\{a_i^k\}_{i\in \hN^+, k=1,2,3,4}$ is simplified to
	\begin{eqnarray*}
	a_i^1 &=& \prod_{j=1}^i\left(1-\frac{2r_j\hat{P}_j^+(\overline{p})}{\underline{v}_j}\right),\\
	a_i^2 &=& \sum_{j=1}^i\frac{2r_j\hat{Q}_j^+(\overline{q})}{\underline{v}_j},\\
	a_i^3 &=& \sum_{j=1}^i\frac{2x_j\hat{P}_j^+(\overline{p})}{\underline{v}_j},\\
	a_i^4 &=& \prod_{j=1}^i\left(1-\frac{2x_j\hat{Q}_j^+(\overline{q})}{\underline{v}_j}\right), \quad i\in \hN^+,
	\end{eqnarray*}
and Condition C1 is simplified to
$$a_{i-1}^1r_i>a_{i-1}^2x_i, ~a_{i-1}^3r_i<a_{i-1}^4x_i \text{ for } i\in \hN^+.$$
		
Lemma \ref{lemma: exact} takes the following form in one-line networks.
\begin{lemma}\label{lemma: line exact}
Consider a one-line network. Assume $f_0$ is strictly increasing and there exists $\overline{p}_i$ and $\overline{q}_i$ such that \eqref{relax constraints} holds for $i\in\hN^+$. Then the SOCP relaxation is exact if C1 holds and all optimal solutions $w=(s, S,\ell,v,s_0)$ of the SOCP relaxation satisfy $\hat{v}_i(s)\leq \overline{v}_i$ for $i\in \hN^+$.
\end{lemma}

\subsection{Proof for one-line networks (Lemma \ref{lemma: line exact})}\label{app: line}
The idea of the proof is as follows. Suppose the conditions in Lemma \ref{lemma: line exact} hold. If there exists an optimal solution $w$ of the SOCP relaxation that violates \eqref{Line 7}, then one can construct another feasible point $w'$ of the SOCP relaxation that has a smaller objective value than $w$, which contradicts with $w$ being optimal. Hence, every solution of the SOCP relaxation must satisfy \eqref{Line 7}, i.e., the SOCP relaxation is exact.

Specifically, given an optimal solution $w=(s,S,\ell,v,s_0)$ of the SOCP relaxation that violates \eqref{Line 7}, one can derive a contradiction through the following steps, which are described in turn.
	\begin{itemize}
	\item[(S1)] Construct a new point $w'(\epsilon)$ according to Algorithm \ref{algorithm} for sufficiently small $\epsilon>0$.

	\item[(S2)] Prove that if $\hat{v}_i(s)\leq \overline{v}_i$ for $i\in \hN^+$ and $w'(\epsilon)$ satisfies \eqref{observation}, then $w'(\epsilon)$ is feasible for the SOCP relaxation. If further, $f_0$ is strictly increasing, then $w'(\epsilon)$ has a smaller objective value than $w$, which contradicts with $w$ being optimal. 

	\item[(S3)] Linearize $w'(\epsilon)$ around $\epsilon=0$ and transform \eqref{observation} into \eqref{key step}, a condition that depends only on $w$ and not on $w'(\epsilon)$. If \eqref{key step} holds, then \eqref{observation} holds for sufficiently small $\epsilon$.

	\item[(S4)] Prove that if there exists $\overline{p}_i$ and $\overline{q}_i$ such that \eqref{relax constraints} holds for $i\in\hN^+$, then C1 implies \eqref{key step}.
	\end{itemize}

\subsection*{Step (S1)}
Let $w=(s,S,\ell,v,s_0)$ be an arbitrary feasible point of the SOCP relaxation that violates \eqref{Line 7}. Define
	\begin{equation}\label{m}
	m:=\min\left\{i\in \hN^+~\left|~\ell_i>|S_i|^2/v_i\right.\right\},
	\end{equation}
then $m\in \hN^+$ and $\ell_m>|S_m|^2/v_m$. For every $\epsilon$ such that $0<\epsilon\leq \ell_m-|S_m|^2/v_m$, we can construct a point $w'(\epsilon)=(s',S',\ell',v',s_0')$ according to Algorithm \ref{algorithm}.

	\begin{algorithm}[!ht]
	\caption{Construct a new point}
	\label{algorithm}
	
	\begin{algorithmic}
	\REQUIRE ~\\
	$w=(s,S,\ell,v,s_0)$, feasible for SOCP but violates \eqref{Line 7};\\
	$m=\min\left\{i\in \hN^+~|~\ell_i>|S_i|^2/v_i\right\}$;\\
	$\epsilon\in\left(0,\ell_m-|S_m|^2/v_m\right]$.
	\ENSURE ~\\
	$w'(\epsilon)=(s',S',\ell',v',s_0')$.
	\end{algorithmic}
	
	\vspace{-0.1in}
	$\line(1,0){242}$
	
	\begin{enumerate}	
	\item Construct $s'$: $s'\leftarrow s$.
	
	\item Construct $S'$, $\ell'$, and $s_0'$ ($S=P+\ii Q$, $S'=P'+\ii Q'$):
	\begin{itemize}
	\item for $k > m$, $\ell_k' \leftarrow \ell_k$.
    	\item for $k \geq m$, $S_k' \leftarrow S_k$.
	\item do the following recursively for $k=m,\ldots,1$:
        \begin{align*}
        & \ell_k' \leftarrow \begin{cases} \ell_k-\epsilon & k=m\\ \frac{\displaystyle \max\left\{P_k'^2,P_k^2\right\} + \max\left\{Q_k'^2,Q_k^2\right\}}{\displaystyle v_k} & k<m,\end{cases} \\
	& \begin{cases}
	S_{k-1}' \leftarrow S_k'-z_k\ell_k'+s_{k-1}' & k\neq1 \\
	s_0' \leftarrow z_1\ell_1' - S_1' & k=1.
	\end{cases}
        \end{align*}
	
	\end{itemize}
	
	\item Construct $v'$:
	\begin{itemize}
	\item set $v_0'\leftarrow v_0$;
	
	\item for $k=1,2,\ldots,n$,
	
	\begin{algorithmic}
	       \STATE $v_k' \leftarrow v_{k-1}'+2\re(z_k^*S_k')-|z_k|^2\ell_k'$.
	\end{algorithmic}
	\end{itemize}
	
	\end{enumerate}
	
	\end{algorithm}

Algorithm \ref{algorithm} keeps $s$ unchanged, i.e., $s'=s$, therefore $w'$ satisfies \eqref{Line 5}. The main step in Algorithm \ref{algorithm} is constructing $S'$, $\ell'$ and $s_0'$, after which $v'$ is simply constructed to satisfy \eqref{Line 2}. After initializing $\ell_k'=\ell_k$ for $k>m$ and $S_k'=S_k$ for $k\geq m$, equation \eqref{Line 6} is satisfied and equation \eqref{Line 4} is satisfied for $i>m$.

The construction of $\ell_k'$ and $S_{k-1}'$ (or $s_0'$) is done recursively for $k=m,\ldots,1$ as follows. First construct $\ell_k'$: reduce $\ell_k$ to $\ell_k'=\ell_k-\epsilon$ if $k=m$; and modify $\ell_k$  so that the constraints in \eqref{line relax} remain satisfied (assuming $v_k'= v_k$) after $S_k$ is changed if $k<m$. After deciding on $\ell_k'$, the construction of $S_{k-1}'$ (or $s_0'$) is simply to satisfy \eqref{Line 4}.

Hence, $w'(\epsilon)$ may only violate \eqref{Line 3} and \eqref{line relax} out of all the constraints in the SOCP relaxation.

\subsection*{Step (S2)}
It suffices to prove that under the conditions given in Lemma \ref{lemma: line exact}, for any optimal solution $w$ of the SOCP relaxation that violates \eqref{Line 7}, there exists $\epsilon>0$ such that $w'(\epsilon)$ is feasible for the SOCP relaxation and has a smaller objective value than $w$. The following lemma gives a sufficient condition under which $w'(\epsilon)$ is indeed feasible and ``better''. To state the lemma, let $S_0\eqdef -s_0$ denote the power that the substation injects to the main grid, and define
    $$\Delta S_i=\Delta P_i+\ii \Delta Q_i:=S_i'-S_i, \quad \Delta v_i:=v_i'-v_i$$
as the difference from $w'$ to $w$ for $i\in \hN$. For complex numbers, let the operators $<$, $\leq$, $>$, $\geq$ denote componentwise.

	\begin{lemma}\label{lemma: observation}
	Consider a one-line network. Given a feasible point $w=(s,S,\ell,v,s_0)$ of the SOCP relaxation that violates \eqref{Line 7}, let $m$ be defined as \eqref{m}, $\epsilon\in(0,\ell_m-|S_m|^2/v_m]$, and $w'(\epsilon)=(s',S',\ell',v',s_0')$ be the output of Algorithm \ref{algorithm}.
    \begin{itemize}
    \item If $\hat{v}_i(s)\leq \overline{v}_i$ for $i\in \hN^+$ and
	\begin{equation}\label{observation}
	S_k' > S_k \text{ for } 0\leq k\leq m-1,
	\end{equation}
then $w'(\epsilon)$ is feasible for the SOCP relaxation.
    \item If further, $f_0$ is strictly increasing, then $w'(\epsilon)$ has a smaller objective value than $w$, i.e.,
    \end{itemize}
    $$\sum_{i=0}^n f_i(\re(s_i')) < \sum_{i=0}^n f_i(\re(s_i)).$$
	\end{lemma}	
	
	\begin{proof}
	As having been discussed in Step (S1), to check that $w'(\epsilon)$ is feasible for the SOCP relaxation, it suffices to show that $w'(\epsilon)$ satisfies \eqref{Line 3} and \eqref{line relax}, i.e., $v_i'\geq\underline{v}_i$, $\ell_i'\geq |S_i'|^2/v_i'$, and $v_i'\leq\overline{v}_i$ for $i\in\hN^+$.
	
	First show that $v_i'\geq\underline{v}_i$ for $i\in\hN^+$.  If \eqref{observation} holds, then $\Delta S_k>0$ for $0\leq k\leq m-1$ and $\Delta S_k=0$ for $m\leq k\leq n$. It follows from \eqref{Line 2} and \eqref{Line 4} that
	\begin{eqnarray*}
	\Delta v_i - \Delta v_{i-1} = \re ( z_i^* \Delta S_i) + \re ( z_i^* \Delta S_{i-1}) \geq 0
	\end{eqnarray*}
for $i \in \hN^+$ and the inequality is strict for $1\leq i\leq m$. Hence,
	\begin{eqnarray*}
	\Delta v_n \geq \ldots \geq \Delta v_m > \Delta v_{m-1} >\ldots >  \Delta v_0=0,
	\end{eqnarray*}
which implies $v_i'> v_i\geq\underline{v}_i$ for $i\in \hN^+$.

Next show that $\ell_i'\geq |S_i'|^2/v_i'$ for $i\in\hN^+$. For $i>m$, one has $\ell_i'=\ell_i\geq|S_i|^2/v_i=|S_i'|^2/v_i\geq|S_i'|^2/v_i'$; for $i=m$, one has $\ell_i'=\ell_i-\epsilon\geq|S_i|^2/v_i=|S_i'|^2/v_i\geq|S_i'|^2/v_i'$; for $i<m$, one has
	\begin{eqnarray*}
	\ell_i' &\!\!\!=& \!\!\!\frac{\max\{P_i'^2,P_i^2\}+\max\{Q_i'^2,Q_i^2\}}{v_i}\\
	&\!\!\!\geq& \!\!\!\frac{P_i'^2+Q_i'^2}{v_i} ~=~ \frac{|S_i'|^2}{v_i} ~\geq~ \frac{|S_i'|^2}{v_i'}.
	\end{eqnarray*}
Hence, $\ell_i'\geq |S_i'|^2/v_i'$ for $i\in\hN^+$.

Finally show that $v_i'\leq \overline{v}_i$ for $i\in\hN^+$. Since $\ell_i'\geq0$ for $i\in\hN^+$, it follows from \eqref{Line 4} that
	\begin{eqnarray*}
	S_{i-1}' = S_i' - z_i\ell_i'+s_{i-1}' \leq S_i' + s_{i-1}'
	\end{eqnarray*}
for $i=2,\ldots,n$. Hence, one has
	\begin{eqnarray*}
	S_i' &\!\!\!\leq& \!\!\!S_{i+1}'+s_i' ~\leq~ S_{i+2}'+s_{i+1}'+s_i' \\
    &\!\!\!\leq& \!\!\!\ldots ~\leq~ S_{n}'+s_{n-1}'+\ldots+s_i' \\
    &\!\!\! =& \!\!\!\sum_{j= i}^n s_j' ~=~ \hat{S}_i(s')
	\end{eqnarray*}
for $i\in\hN^+$. Therefore, it follows from \eqref{Line 2} that
 	\begin{eqnarray*}
	v_i'-v_{i-1}' &=& 2\re(z_i^*S_i') -|z_i|^2\ell_i' \\
	&\leq& 2\re(z_i^*S_i') \\
	&\leq& 2\re(z_i^*\hat{S}_i(s'))
	\end{eqnarray*}
for $i\in\hN^+$. Now, sum up the inequality over $\hP_i$ to obtain
	\begin{eqnarray*}
	v_i'-v_0' \leq 2\sum_{j=1}^i\re(z_j^*\hat{S}_j(s'))
	\end{eqnarray*}
for $i\in \hN^+$, which implies
    $$v_i'\leq v_0 + 2\sum_{j=1}^i\re(z_j^*\hat{S}_j(s')) =\hat{v}_i(s') =\hat{v}_i(s) \leq \overline{v}_i$$
for $i\in \hN^+$.

To this end, it has been proved that if $\hat{v}_i(s)\leq\overline{v}_i$ for $i\in\hN^+$ and \eqref{observation} holds, then $w'(\epsilon)$ is feasible for the SOCP relaxation.

If further, $f_0$ is strictly increasing, then one has
	\begin{eqnarray*}
	&& \sum_{i=0}^n f_i(\re(s_i')) - \sum_{i=0}^n f_i(\re(s_i)) \\
	&=& f_0(\re(s_0')) - f_0(\re(s_0))\\
    	&=& f_0(-\re(S_0')) - f_0(-\re(S_0)) < 0,
	\end{eqnarray*}
i.e., $w'(\epsilon)$ has a smaller objective value than $w$. This completes the proof of Lemma \ref{lemma: observation}.
	\end{proof}
	
	\begin{figure}[!htbp]
     	\centering
     	\includegraphics[scale=0.25]{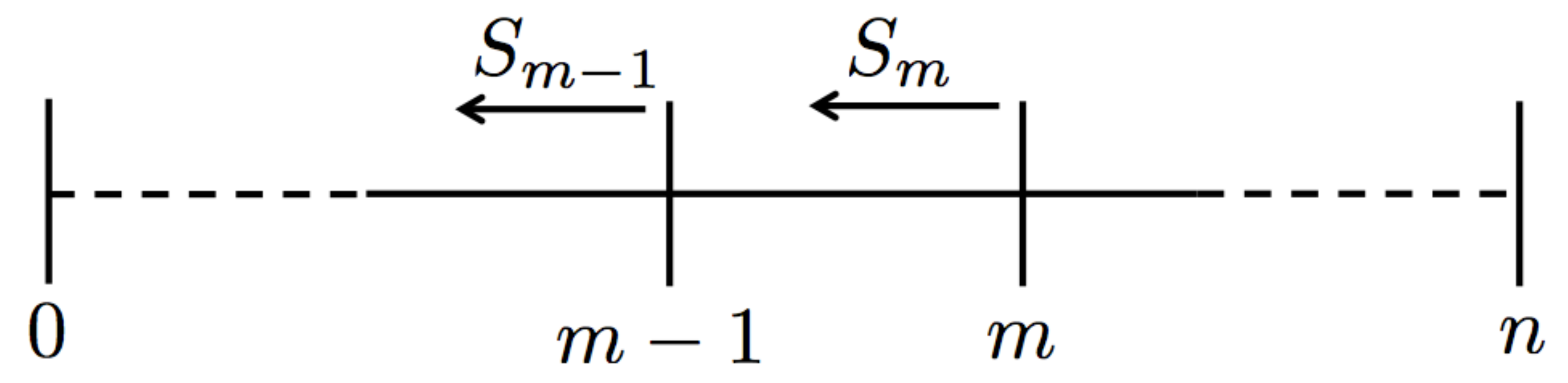}
      	\caption{Illustration of \eqref{observation}.}
      	\label{fig: illustrate}
	\end{figure}

We illustrate that \eqref{observation} seems natural to hold with Fig. \ref{fig: illustrate}. Recall that $S_m'=S_m$ and $\ell_m'=\ell_m-\epsilon$, therefore
	\begin{eqnarray}
	\Delta S_{m-1} = \Delta S_m-z_m\Delta\ell_m = z_m\epsilon > 0. \label{m difference}
	\end{eqnarray}
Intuitively, after increasing $S_{m-1}$, upstream power flow ($S_k$ for $0\leq k\leq m-2$) is likely to increase. Lemma \ref{lemma: observation} says that if upstream power flow indeed increases, then $w'(\epsilon)$ is feasible and ``better'' than $w$.

Assume that $f_0$ is strictly increasing. Lemma \ref{lemma: observation} implies that for any optimal solution $w=(s,S,\ell,v,s_0)$ of the SOCP relaxation, if $\hat{v}_i(s)\leq \overline{v}_i$ for $i\in \hN^+$ and there exists some $\epsilon\in(0,\ell_m-|S_m|^2/v_m]$ such that $w'(\epsilon)$ satisfies \eqref{observation}, then $w$ cannot be optimal for the SOCP relaxation, which is a contradiction.

\subsection*{Step (S3)}
We transform \eqref{observation}, which depends on both $w$ and $w'$, to \eqref{key step}, that only depends on $w$, in this step. The idea is to approximate $w'(\epsilon)$ by its Taylor expansion near $\epsilon=0$.

First compute the Taylor expansion of $w'(\epsilon)$. It follows from \eqref{m difference} that
 	\begin{equation}\label{initial}
    \begin{pmatrix} \Delta P_{m-1}\\ \Delta Q_{m-1} \end{pmatrix}
	= \begin{pmatrix} r_m\\ x_m \end{pmatrix} \epsilon > 0.
    \end{equation}
For any $t\in\{1,\ldots,m-1\}$, if
	\begin{equation}\label{recursive condition}
    \begin{pmatrix} \Delta P_t\\ \Delta Q_t \end{pmatrix} = B_t\epsilon + O(\epsilon^2)
    \end{equation}
for some $B_t>0$, then $(\Delta P_t,\Delta Q_t)^T>0$ for sufficiently small $\epsilon$. It follows that
	\begin{eqnarray*}
	&&\!\!\! \!\!\! \!\!\! \! \Delta \ell_t \,=\, \frac{\max\{P_t'^2,P_t^2\}+\max\{Q_t'^2,Q_t^2\}}{v_t}-\frac{P_t^2+Q_t^2}{v_t}\\
	&\!\!\!=& \!\!\!\frac{\max\{P_t'^2-P_t^2,0\}+\max\{Q_t'^2-Q_t^2,0\}}{v_t}\\
	&\!\!\!=& \!\!\!\frac{\max\{2P_t\Delta P_t+O(\epsilon^2),0\}}{v_t}  +\frac{\max\{2Q_t\Delta Q_t+O(\epsilon^2),0\}}{v_t}\\
	&\!\!\!=& !\!\!\frac{2P_t^+\Delta P_t}{v_t} +\frac{2Q_t^+\Delta Q_t}{v_t}+O(\epsilon^2)
	\end{eqnarray*}
and
	\begin{eqnarray}
	&& \!\!\!\!\!\!\!\!\!\!\!\!\!\!\!\!\!\!\!\! \begin{pmatrix} \Delta P_{t-1}\\ \Delta Q_{t-1} \end{pmatrix} \,=\,
	\begin{pmatrix} \Delta P_t\\ \Delta Q_t \end{pmatrix}
	-\begin{pmatrix} r_t\\ x_t \end{pmatrix}\Delta \ell_t \nonumber\\
	&=& \!\!\! \begin{pmatrix} 1-\frac{2r_tP_t^+}{v_t} & -\frac{2r_tQ_t^+}{v_t}\\ -\frac{2x_tP_t^+}{v_t} & 1-\frac{2x_tQ_t^+}{v_t} \end{pmatrix}
	\begin{pmatrix} \Delta P_t\\ \Delta Q_t \end{pmatrix} +O(\epsilon^2).\label{recursive}
	\end{eqnarray}
With this recursive relation and the initial value in \eqref{initial}, the Taylor expansion of $(\Delta P_t,\Delta Q_t)^T$ near $\epsilon=0$ can be computed for $0\leq t \leq m-1$, as given in Lemma \ref{lemma: Taylor}.
	
To state the lemma, given any feasible point $w=(s,S=P+\ii Q,\ell,v,s_0)$ of the SOCP relaxation, define $c_i(w):=1-2r_iP_i^+/v_i$, $d_i(w):=2r_iQ_i^+/v_i$, $e_i(w):=2x_iP_i^+/v_i$, $f_i(w):=1-2x_iQ_i^+/v_i$, and
	\begin{eqnarray*}
	A_i(w)=\begin{pmatrix} c_i(w) & -d_i(w)\\ -e_i(w) & f_i(w) \end{pmatrix}
	\end{eqnarray*}
for $i\in \hN^+$. Also define $u_i:=(r_i,x_i)^T$ for $i\in\hN^+$.

    \begin{lemma}\label{lemma: Taylor}
    Consider a one-line network. Given a feasible point $w=(s,S,\ell,v,s_0)$ of the SOCP relaxation that violates \eqref{Line 7}, let $m$ be defined as \eqref{m}, $\epsilon\in(0,\ell_m-|S_m|^2/v_m]$, and $w'(\epsilon)=(s',S',\ell',v',s_0')$ be the output of Algorithm \ref{algorithm}. If
	\begin{equation*}
	A_t(w)\cdots A_{m-1}(w)u_m > 0
	\end{equation*}
for $t=1,\ldots,m$, then
	\begin{equation}\label{Taylor}
	(\Delta P_{t-1}, \Delta Q_{t-1})^T = A_t\cdots A_{m-1}u_m\epsilon+O(\epsilon^2)
	\end{equation}
for $t=1,\ldots,m$.
    \end{lemma}
    \begin{proof}
    We prove that \eqref{Taylor} holds for $t=m,\ldots,1$ by mathematical induction on $t$.
    \begin{itemize}
    \item[i)] When $t=m$, one has $(\Delta P_{t-1}, \Delta Q_{t-1})^T = u_m\epsilon$ according to \eqref{initial}. Hence, \eqref{Taylor} holds for $t=m$.
    \item[ii)] Assume that \eqref{Taylor} holds for $t=k$ ($2\leq k\leq m$), i.e., $$(\Delta P_{k-1}, \Delta Q_{k-1})^T = A_k\cdots A_{m-1}u_m\epsilon+O(\epsilon^2).$$ Since $A_k\cdots A_{m-1}u_m>0$, equation \eqref{recursive condition} holds for $t=k-1$. It follows that when $t=k-1$, one has
        \begin{eqnarray*}
	    \begin{pmatrix} \Delta P_{t-1}\\ \Delta Q_{t-1} \end{pmatrix}
	    \!\!\!&=& \!\!\! A_t \begin{pmatrix} \Delta P_t\\ \Delta Q_t \end{pmatrix} +O(\epsilon^2) \\
        \!\!\!&=& \!\!\! A_t \left( A_k\cdots A_{m-1}u_m\epsilon+O(\epsilon^2) \right) +O(\epsilon^2) \\
        \!\!\!&=& \!\!\! A_t \cdots A_{m-1}u_m\epsilon +O(\epsilon^2),
	    \end{eqnarray*}
i.e., equation \eqref{Taylor} holds for $t=k-1$.
    \end{itemize}
According to (i) and (ii), equation \eqref{Taylor} holds for $t=m,\ldots,1$. This completes the proof of Lemma \ref{lemma: Taylor}.
    \end{proof}

The condition in \eqref{observation} can be simplified to a new condition that only depends on $w$, as given in the following corollary, using the Taylor expansion of $(\Delta P,\Delta Q)$ given in Lemma \ref{lemma: Taylor}.
	\begin{corollary}\label{lemma: key lemma}
Consider a one-line network. Let $w$ be a feasible point of the SOCP relaxation that violates \eqref{Line 7}. If
	\begin{equation}\label{key step}
	A_j(w)\cdots A_{i-1}(w)u_i>0, \quad 1\leq j \leq  i\leq n,
	\end{equation}
then the output $w'(\epsilon)$ of Algorithm \ref{algorithm} satisfies \eqref{observation} for sufficiently small $\epsilon$.
	\end{corollary}
    \begin{proof}
    Let $m$ be defined as \eqref{m} and substitute $i=m$ in \eqref{key step} to obtain
        $$A_j(w)\cdots A_{m-1}(w)u_m>0, \quad 1\leq j \leq  m.$$
    Then, it follows from Lemma \ref{lemma: Taylor} that
        $$(\Delta P_{t-1}, \Delta Q_{t-1})^T = A_t\cdots A_{m-1}u_m\epsilon+O(\epsilon^2)$$
    for $t=1,\ldots,m$. Hence, $\Delta S_k>0$ for $k=0,\ldots,m-1$, for sufficiently small $\epsilon$.
    \end{proof}

Corollary \ref{lemma: key lemma} implies that if there exists an optimal solution $w=(s,S,\ell,v,s_0)$ of the SOCP relaxation that violates \eqref{Line 7} but satisfies \eqref{key step}, then $w'(\epsilon)$ satisfies \eqref{observation} for sufficiently small $\epsilon$. If further, $w$ satisfies $\hat{v}_i(s)\leq\overline{v}_i$ for $i\in \hN^+$, then $w'(\epsilon)$ is feasible for the SOCP relaxation (Lemma \ref{lemma: observation}). If further, $f_0$ is strictly increasing, then $w'(\epsilon)$ has a smaller objective value than $w$ (Lemma \ref{lemma: observation}). This contradicts with $w$ being optimal.

	\begin{corollary}\label{corollary: key}
    Consider a one-line network. Assume that $f_0$ is strictly increasing; every optimal solution $w=(s,S,\ell,v,s_0)$ of the SOCP relaxation satisfies $\hat{v}_i(s)\leq\overline{v}_i$ for $i\in\hN^+$ and \eqref{key step}. Then the SOCP relaxation is exact.
	\end{corollary}
    \begin{proof}
    Assume the SOCP relaxation is not exact, then one can derive a contradiction as follows.

    Since the SOCP relaxation is not exact, there exists an optimal solution $w=(s,S,\ell,v,s_0)$ of the SOCP relaxation that violates \eqref{Line 7}. Since $w$ satisfies \eqref{key step}, one can pick a sufficiently small $\epsilon>0$ such that $w'(\epsilon)$ satisfies \eqref{observation} according to Corollary \ref{lemma: key lemma}. Further, the point $w'(\epsilon)$ is feasible for the SOCP relaxation and has a smaller objective value than $w$, according to Lemma \ref{lemma: observation}. This contradicts with $w$ being optimal for the SOCP relaxation.

    Hence, the SOCP relaxation is exact.
    \end{proof}

\subsection*{Step (S4)}
To complete the proof of Lemma \ref{lemma: line exact}, it suffices to prove the following lemma.
	\begin{lemma}\label{lemma: main condition}
	Consider a one-line network. If there exists $\overline{p}_i$ and $\overline{q}_i$ such that \eqref{relax constraints} holds for $i\in\hN^+$ and C1 holds, then \eqref{key step} holds for all feasible points of the SOCP relaxation.
	\end{lemma}
Before proving Lemma \ref{lemma: main condition}, we first show how Lemma \ref{lemma: line exact} follows from Lemma \ref{lemma: main condition}.

\noindent{\it Proof of Lemma \ref{lemma: line exact}. }
If the conditions in Lemma \ref{lemma: line exact} hold, then the conditions in Lemma \ref{lemma: main condition} hold. It follows from Lemma \ref{lemma: main condition} that \eqref{key step} holds for all optimal solution of the SOCP relaxation.

Then, the conditions in Corollary \ref{corollary: key} hold, and it follows that the SOCP relaxation is exact. This completes the proof of Lemma \ref{lemma: line exact}. $\hfill\Box$

Now we come back to Lemma \ref{lemma: main condition}. To prove it, we need the following lemma.
    \begin{lemma}\label{lemma: matrix multiplication}
	Given $i\geq1$; $c$, $d$, $e$, $f\in\mathbb{R}^i$ such that $0<c\leq1$, $d\geq0$, $e\geq0$, $0<f\leq1$ componentwise; and $u\in\mathbb{R}^2$ that satisfies $u>0$. If
	\begin{equation}\label{lemma 7 condition}
	\begin{pmatrix}
	\displaystyle \prod_{j=1}^i c_j & \displaystyle -\sum_{j=1}^i d_j\\
	\displaystyle -\sum_{j=1}^i e_j & \displaystyle \prod_{j=1}^i f_j
	\end{pmatrix}
	u>0,
	\end{equation}
then
    \begin{equation}\label{lemma 7 result}
    \begin{pmatrix}
	c_j & -d_j\\
	-e_j & f_j
	\end{pmatrix}\cdots
    \begin{pmatrix}
	c_i & -d_i\\
	-e_i & f_i
	\end{pmatrix} u > 0
    \end{equation}
for $j=1,\ldots,i$.
	\end{lemma}
	
	\begin{proof}
We prove Lemma \ref{lemma: matrix multiplication} by mathematical induction on $i$.
	\begin{itemize}
	\item[i)] When $i=1$, Lemma \ref{lemma: matrix multiplication} holds trivially.
	\item[ii)] Assume that Lemma \ref{lemma: matrix multiplication} holds for $i=K$ ($K\geq1$), and we will show that Lemma \ref{lemma: matrix multiplication} also holds for $i=K+1$. When $i=K+1$, if \eqref{lemma 7 condition} holds, i.e.,
    \begin{eqnarray*}
	\begin{pmatrix}
	\displaystyle \prod_{j=1}^{K+1}c_j & \displaystyle -\sum_{j=1}^{K+1}d_j\\
	\displaystyle -\sum_{j=1}^{K+1}e_j & \displaystyle \prod_{j=1}^{K+1}f_j
	\end{pmatrix}
	u > 0,
	\end{eqnarray*}
we prove that \eqref{lemma 7 result} holds for $j=1,\ldots,K+1$ as follows.

First prove that \eqref{lemma 7 result} holds for $j=2,\ldots,K+1$. The idea is to construct some $c',d',e',f'\in\mathbb{R}^K$ and apply the induction hypothesis. The construction is
    \begin{eqnarray*}
    c' &=& (c_2,~c_3,~\ldots,~c_{K+1}),\\
    d' &=& (d_2,~d_3,~\ldots,~d_{K+1}),\\
    e' &=& (e_2,~e_3,~\ldots,~e_{K+1}),\\
    f' &=& (f_2,~f_3,~\ldots,~f_{K+1}).
    \end{eqnarray*}
Clearly, $c',d',e',f'$ satisfies $0<c'\leq1$, $d'\geq0$, $e'\geq0$, $0<f'\leq1$ componentwise and
	\begin{eqnarray*}
	\begin{pmatrix}
	\displaystyle \prod_{j=1}^Kc_j' & \displaystyle -\sum_{j=1}^Kd_j'\\
	\displaystyle -\sum_{j=1}^Ke_j' & \displaystyle \prod_{j=1}^Kf_j'
	\end{pmatrix}
	u
    &\!\!\!=& \!\!\!\begin{pmatrix}
	\displaystyle \prod_{j=2}^{K+1}c_j & \displaystyle -\sum_{j=2}^{K+1}d_j\\
	\displaystyle -\sum_{j=2}^{K+1}e_j & \displaystyle \prod_{j=2}^{K+1}f_j
	\end{pmatrix}
	u\\
	&\!\!\!\!\!\!\!\!\!\!\!\!\!\!\!\!\!\!\!\!\!\!\!\!\!\!\!\!\!\!\geq&
	\!\!\!\!\!\!\!\!\!\!\!\!\!\!\!\!\!\!\!\!\begin{pmatrix}
	\displaystyle \prod_{j=1}^{K+1}c_j & \displaystyle -\sum_{j=1}^{K+1}d_j\\
	\displaystyle -\sum_{j=1}^{K+1}e_j & \displaystyle \prod_{j=1}^{K+1}f_j
	\end{pmatrix}
	u > 0.
	\end{eqnarray*}
Apply the induction hypothesis to obtain that
    $$\begin{pmatrix}
	c_j' & -d_j'\\
	-e_j' & f_j'
	\end{pmatrix}\cdots
    \begin{pmatrix}
	c_K' & -d_K'\\
	-e_K' & f_K'
	\end{pmatrix} u > 0$$
for $j=1,\ldots,K$, i.e., \eqref{lemma 7 result} holds for $j=2,\ldots,K+1$.

Next prove that \eqref{lemma 7 result} holds for $j=1$. The idea is still to construct some $c',d',e',f'\in\mathbb{R}^K$ and apply the induction hypothesis. The construction is
    \begin{eqnarray*}
    c' &=& (c_1c_2,~c_3,~\ldots,~c_{K+1}),\\
    d' &=& (d_1+d_2,~d_3,~\ldots,~d_{K+1}),\\
    e' &=& (e_1+e_2,~e_3,~\ldots,~e_{K+1}),\\
    f' &=& (f_1f_2,~f_3,~\ldots,~f_{K+1}).
    \end{eqnarray*}
Clearly, $c',d',e',f'$ satisfies $0<c'\leq1$, $d'\geq0$, $e'\geq0$, $0<f'\leq1$ componentwise and
    \begin{eqnarray*}
	\begin{pmatrix}
	\displaystyle \prod_{j=1}^Kc_j' & \displaystyle -\sum_{j=1}^Kd_j'\\
	\displaystyle -\sum_{j=1}^Ke_j' & \displaystyle \prod_{j=1}^Kf_j'
	\end{pmatrix}
	u
    &\!\!\!=& \!\!\!\begin{pmatrix}
	\displaystyle \prod_{j=1}^{K+1}c_j & \displaystyle -\sum_{j=1}^{K+1}d_j\\
	\displaystyle -\sum_{j=1}^{K+1}e_j & \displaystyle \prod_{j=1}^{K+1}f_j
	\end{pmatrix}
	u\\
	&\!\!\!>& \!\!\!0.
	\end{eqnarray*}
Apply the induction hypothesis to obtain that
    \begin{eqnarray*}
    v_2':=\begin{pmatrix}
	c_2' & -d_2'\\
	-e_2' & f_2'
	\end{pmatrix}\cdots
    \begin{pmatrix}
	c_K' & -d_K'\\
	-e_K' & f_K'
	\end{pmatrix} u > 0,\\
    v_1':=\begin{pmatrix}
	c_1' & -d_1'\\
	-e_1' & f_1'
	\end{pmatrix}\cdots
    \begin{pmatrix}
	c_K' & -d_K'\\
	-e_K' & f_K'
	\end{pmatrix} u > 0.
    \end{eqnarray*}
It follows that
    \begin{eqnarray*}
    && \!\!\!\!\!\!\!\!\!\!\!\!\!\!\!\!\!\!\!\!\begin{pmatrix}
	c_1 & -d_1\\
	-e_1 & f_1
	\end{pmatrix}\cdots
    \begin{pmatrix}
	c_{K+1} & -d_{K+1}\\
	-e_{K+1} & f_{K+1}
	\end{pmatrix} u \\
    &=& \begin{pmatrix}
	c_1 & -d_1\\
	-e_1 & f_1
	\end{pmatrix}
    \begin{pmatrix}
	c_2 & -d_2\\
	-e_2 & f_2
	\end{pmatrix} \\
    && \quad
    \begin{pmatrix}
	c_3 & -d_3\\
	-e_3 & f_3
	\end{pmatrix}
\cdots
    \begin{pmatrix}
	c_{K+1} & -d_{K+1}\\
	-e_{K+1} & f_{K+1}
	\end{pmatrix} u \\
    &=& \begin{pmatrix}
	c_1 & -d_1\\
	-e_1 & f_1
	\end{pmatrix}
    \begin{pmatrix}
	c_2 & -d_2\\
	-e_2 & f_2
	\end{pmatrix} \\
    && \quad
    \begin{pmatrix}
	c_2' & -d_2'\\
	-e_2' & f_2'
	\end{pmatrix}
\cdots
    \begin{pmatrix}
	c_K' & -d_K'\\
	-e_K' & f_K'
	\end{pmatrix} u \\
    &=& \begin{pmatrix}
	c_1 & -d_1\\
	-e_1 & f_1
	\end{pmatrix}
    \begin{pmatrix}
	c_2 & -d_2\\
	-e_2 & f_2
	\end{pmatrix} v_2' \\
    &=& \begin{pmatrix}
	c_1c_2+d_1e_2 & -c_1d_2-d_1f_2\\
	-e_1c_2-f_1e_2 & f_1f_2+e_1d_2
	\end{pmatrix} v_2' \\
    &\geq& \begin{pmatrix}
	c_1c_2 & -d_2-d_1\\
	-e_1-e_2 & f_1f_2
	\end{pmatrix} v_2' \\
    &=& \begin{pmatrix}
	c_1' & -d_1'\\
	-e_1' & f_1'
	\end{pmatrix} v_2' \\
&=& v_1'>0,
    \end{eqnarray*}
i.e., \eqref{lemma 7 result} holds for $j=1$.

To this end, we have proved that \eqref{lemma 7 result} holds for $j=1,\ldots,K+1$, i.e., Lemma \ref{lemma: matrix multiplication} also holds for $i=K+1$.
	\end{itemize}
According to (i) and (ii), Lemma \ref{lemma: matrix multiplication} holds for $i\geq1$.
	\end{proof}

In the end, we prove Lemma \ref{lemma: main condition} with the result in Lemma \ref{lemma: matrix multiplication}.

\noindent{\it Proof of Lemma \ref{lemma: main condition}.}
    Assume that there exists $\overline{p}_i$ and $\overline{q}_i$ such that \eqref{relax constraints} holds for $i\in\hN^+$; and that C1 holds. According to Lemma \ref{lemma: matrix multiplication}, it suffices to prove that for any feasible point $w=(s=p+\ii q,S=P+\ii Q,\ell,v)$ of the SOCP relaxation, the following inequalities hold:
    \begin{itemize}
    \item[i)] $0<c_j(w)\leq 1, d_j(w)\geq 0, e_j(w)\geq 0, 0<f_j(w)\leq1$ for $j=1,\ldots,n-1$;
    \item[ii)]
    \[\begin{pmatrix}
	\displaystyle \prod_{j=1}^{i-1} c_j(w) & \displaystyle  -\sum_{j=1}^{i-1}d_j(w)\\
	\displaystyle  -\sum_{j=1}^{i-1} e_j(w) & \displaystyle  \prod_{j=1}^{i-1}f_j(w)
	\end{pmatrix}
	\begin{pmatrix}
	r_i \\ x_i
	\end{pmatrix}
	> 0\]
for $i=2,\ldots, n$.
    \end{itemize}

	First check (i). It follows from $a_{i-1}^1r_i>a_{i-1}^2x_i\geq0$ (see C1) for $i\in \hN^+$ that $a_{i-1}^1>0$ for $i\in \hN^+$. Then, it follows from
    $$0<a_{i-1}^1=\prod_{j=1}^{i-1}\left( 1-\frac{2r_j\hat{P}_j^+(\overline{p})}{\underline{v}_j} \right)$$
    for $i=1,\ldots,n$ that
    $$1-2r_j\hat{P}_j^+(\bar{p})/\underline{v}_j>0$$
for $j=1,\ldots,n-1$. Since $\ell_i\geq0$ for $i\in\hN^+$, it follows from \eqref{Line 4} that
	\begin{eqnarray*}
	S_{i-1} = S_i - z_i\ell_i+s_{i-1} \leq S_i + s_{i-1}
	\end{eqnarray*}
for $i\in\hN^+$. Hence, one has
	\begin{eqnarray*}
	S_i &\!\!\!\leq& \!\!\!S_{i+1}+s_i ~\leq~ S_{i+2}+s_{i+1}+s_i \\
    &\!\!\!\leq& \!\!\!\ldots ~\leq~ S_{n}+s_{n-1}+\ldots+s_i \\
    &\!\!\! =& \!\!\!\sum_{j= i}^n s_j ~=~ \hat{S}_i(s)
	\end{eqnarray*}
for $i\in\hN$. It follows that $P_j\leq \hat{P}_j(p)\leq \hat{P}_j(\overline{p})$ and therefore $P_j^+\leq \hat{P}_j^+(p)\leq \hat{P}_j^+(\overline{p})$ for $j\in\hN$. It follows that
	\begin{eqnarray*}
	&& c_j(w) = 1-\frac{2r_jP_j^+}{v_j}\in \left[1-\frac{2r_j\hat{P}_j^+(\bar{p})}{\underline{v}_j},1\right]\subseteq \left(0,1\right],\\
	&& e_j(w) = \frac{2x_jP_j^+}{v_j}\in \left[0,\frac{2x_j\hat{P}_j^+(\bar{p})}{\underline{v}_j}\right] \subseteq \left[0,\infty\right)
	\end{eqnarray*}
for $j=1,\ldots,n-1$. Similarly,
	\begin{eqnarray*}
	&& d_j(w) \in \left[ 0, \frac{2r_j\hat{Q}_j^+(\bar{q})}{\underline{v}_j}\right]\subseteq \left[0,\infty\right),\\
	&& f_j(w) \in \left[1-\frac{2x_j\hat{Q}_j^+(\bar{q})}{\underline{v}_j}, 1\right]\subseteq \left(0,1\right]
	\end{eqnarray*}
for $j=1,\ldots,n-1$. Hence, the inequalities in (i) hold.

Next check (ii). One has
    \begin{eqnarray*}
	\prod_{j=1}^{i-1} c_j(w)r_i
    &=& \prod_{j=1}^{i-1} \left(1-\frac{2r_jP_j^+}{v_j}\right) r_i \\
    &\geq& \prod_{j=1}^{i-1} \left(1-\frac{2r_j\hat{P}_j^+(\bar{p})}{\underline{v}_j}\right) r_i \\
    &=& a_{i-1}^1r_i \\
    &>& a_{i-1}^2x_i \qquad \text{(Condition C1)}\\
    &=& \sum_{j=1}^{i-1}\frac{2r_j\hat{Q}_j^+(\overline{q})}{\underline{v}_j}x_i \\
    &\geq& \sum_{j=1}^{i-1}\frac{2r_jQ_j^+}{\underline{v}_j}x_i \\
    &=& \sum_{j=1}^{i-1}d_j(w)x_i
    \end{eqnarray*}
for $i=2,\ldots,n$. Similarly,
    \begin{eqnarray*}
	\sum_{j=1}^{i-1} e_j(w)r_i < \prod_{j=1}^{i-1}f_j(w)x_i
	\end{eqnarray*}
for $i=2,\ldots,n$. Hence, the inequalities in (ii) hold.

To this end, we have proved that the inequalities in (i) and (ii) hold. Apply Lemma \ref{lemma: matrix multiplication} to obtain
    \[\begin{pmatrix}
	c_j(w) & -d_j(w)\\
	-e_j(w) & f_j(w)
	\end{pmatrix}
    \cdots
    \begin{pmatrix}
	c_{i-1}(w) & -d_{i-1}(w)\\
	-e_{i-1}(w) & f_{i-1}(w)
	\end{pmatrix}
	\begin{pmatrix}
	r_i \\ x_i
	\end{pmatrix}
	> 0\]
for $2\leq i\leq n$ and $1\leq j<i$, i.e.,
    $$A_j(w)\cdots A_{i-1}(w)u_i>0$$
for $1\leq j<i\leq n$. When $i=j=1,\ldots,n$, one has
    $$A_j(w)\cdots A_{i-1}(w)u_i = u_i>0.$$
Hence, $A_j(w)\cdots A_{i-1}(w)u_i>0$ for $1\leq j\leq i\leq n$. This completes the proof of Lemma \ref{lemma: main condition}. $\hfill\Box$

\subsection{Proof for tree networks (Lemma \ref{lemma: exact})}\label{app: tree}
To this end, we have proved Lemma \ref{lemma: exact} for one-line networks (Lemma \ref{lemma: line exact}). The proof for general tree networks is similar.

If the SOCP relaxation is not exact, then there exists an optimal solution $w=(s,S,\ell,v,s_0)$ of the SOCP relaxation that violates \eqref{BFM 4}. A contradiction can be derived through the same steps as for one-line networks, which are described in turn.
	\begin{itemize}
	\item[(S1')] Construct a new point $w'(\epsilon)$ according to Algorithm \ref{algorithm tree} for sufficiently small $\epsilon>0$.

	\item[(S2')] Prove that if $\hat{v}_i(s)\leq \overline{v}_i$ for $i\in \hN^+$ and $w'(\epsilon)$ satisfies \eqref{observation tree}, then $w'(\epsilon)$ is feasible for the SOCP relaxation. If further, $f_0$ is strictly increasing, then $w'(\epsilon)$ has a smaller objective value than $w$, which contradicts with $w$ being optimal.

	\item[(S3')] Linearize $w'(\epsilon)$ around $\epsilon=0$ and transform \eqref{observation tree} into \eqref{key step tree}, a condition that depends only on $w$ and not on $w'(\epsilon)$. If \eqref{key step tree} holds, then \eqref{observation tree} holds for sufficiently small $\epsilon$.

	\item[(S4')] Prove that if there exists $\overline{p}_i$ and $\overline{q}_i$ such that \eqref{relax constraints} holds for $i\in\hN^+$, then C1 implies \eqref{key step tree}.
	\end{itemize}

\subsection*{Step (S1')}
Let $w=(s,S,\ell,v,s_0)$ be an arbitrary feasible point of the SOCP relaxation that violates \eqref{BFM 4}, then there exists some line $m\rightarrow \beta$ such that
	 \begin{equation}\label{m tree}
	 \begin{cases}
	 \ell_{m\beta}>|S_{m\beta}|^2/v_m;\\
	 \ell_{kl}=|S_{kl}|^2/v_k, \quad (k,l)\in\hP_\beta.
	 \end{cases}
	 \end{equation}
For every $\epsilon$ such that $0<\epsilon\leq \ell_{m\beta}-|S_{m\beta}|^2/v_m$, one can construct a point $w'(\epsilon)=(s',S',\ell',v',s_0')$ according to Algorithm \ref{algorithm tree}.

To state the algorithm, assume that the buses on path $\mathcal{P}_m$ are indexed $m,m-1,\ldots,0$ without loss of generality, i.e.,
    \begin{equation}\label{path}
    \mathcal{P}_m=\{(m,m-1),\ldots,(1,0)\}.
    \end{equation}
Then $\beta = m-1$. 

	\begin{algorithm}[!h]
	\caption{Construct a new point}
	\label{algorithm tree}
	
	\begin{algorithmic}
	\REQUIRE ~\\
	$w=(s,S,\ell,v,s_0)$, feasible for SOCP but violates \eqref{BFM 4};\\
	$m$ such that $\ell_{m,m-1}>|S_{m, m-1}|^2/v_m$ and $\ell_{k,k-1}=|S_{k,k-1}|^2/v_k$ for $1\leq k<m$;\\
	$\epsilon\in\left(0,\ell_{m,m-1}-|S_{m,m-1}|^2/v_m\right]$.
	\ENSURE ~\\
	$w'(\epsilon)=(s',S',\ell',v',s_0')$.
	\end{algorithmic}
	
	\vspace{-0.1in}
	$\line(1,0){242}$
	
	\begin{enumerate}	
	\item Construct $s'$: $s'\leftarrow s$.
	
	\item Construct $S'$, $\ell'$, and $s_0'$ ($S=P+\ii Q$, $S'=P'+\ii Q'$):
	\begin{itemize}
	\item for $(k,l)\notin\mathcal{P}_m$, $\ell_{kl}' \leftarrow \ell_{kl}$.
    	\item for $(k,l)\notin\mathcal{P}_{m-1}$, $S_{kl}' \leftarrow S_{kl}$.
	\item for $(k,l)\in\hP_m$, do the following recursively for $k=m,\ldots,1$:
        \begin{align*}
        & \!\!\!\!\!\!\!\!\!\!\!\!\!\!\!\!\!\!\!\! \ell_{k,k-1}' \leftarrow \begin{cases}\ell_{k,k-1}-\epsilon & k=m\\ \frac{\footnotesize \max\left\{P_{k,k-1}'^2,P_{k,k-1}^2\right\} + \max\left\{Q_{k,k-1}'^2,Q_{k,k-1}^2\right\}}{\displaystyle v_k} & k<m,\end{cases} \\
	& \!\!\!\!\!\!\!\!\!\!\!\!\!\!\!\!\!\!\!\!\begin{cases} S_{k-1,k-2}' \leftarrow S_{k,k-1}'-z_{k,k-1}\ell_{k,k-1}'+s_{k-1}' & k\neq1 \\
        s_0' \leftarrow z_{10}\ell_{10}' - S_{10}' & k=1.
        \end{cases}
        \end{align*}
	
	\end{itemize}
	
	\item Construct $v'$:
	\begin{itemize}
	\item set $v_0'\leftarrow v_0$, $\hN_\mathrm{set}=\{0\}$, $\hN_\mathrm{unset}=\{1,\ldots,n\}$;
	
	\item while $\hN_\mathrm{unset}\neq\emptyset$,
	\begin{algorithmic}
    \STATE pick a $(k,l)$ such that $l\in\hN_\mathrm{set}$ and $k\in\hN_\mathrm{unset}$;
	\STATE $v_k' \leftarrow v_l'+2\re(z_{kl}^*S_{kl}')-|z_{kl}|^2\ell_{kl}'$;
    \STATE $\hN_\mathrm{set}=\hN_\mathrm{set}\cup\{k\}$;
    \STATE $\hN_\mathrm{unset}=\hN_\mathrm{unset}\backslash\{k\}$;
	\end{algorithmic}
	\end{itemize}
	
	\end{enumerate}
	
	\end{algorithm}

Algorithm \ref{algorithm tree} extends Algorithm \ref{algorithm} to general tree networks. It keeps $s$ unchanged, i.e., $s'=s$, therefore $w'$ satisfies \eqref{BFM 3}. The main step in Algorithm \ref{algorithm tree} is constructing $S'$, $\ell'$, and $s_0'$, after which $v'$ is simply constructed to satisfy \eqref{BFM 1}.

After initializing $\ell_{kl}' = \ell_{kl}$ for $(k,l)\notin\hP_m$ and $S_{kl}' = S_{kl}$ for $(k,l)\notin\mathcal{P}_{m-1}$, \eqref{BFM 2} is satisfied for $i\notin\hP_{m-1}$. Equation \eqref{BFM 2} is also satisfied for $i<m$ since $S_{k-1}'$ (or $s_0'$) are then set recursively for $k=m,\ldots,1$ to satisfy \eqref{BFM 2}.

The construction of $\ell_{kl}'$ on $\hP_m$ is the same as for one-line networks: reduce $\ell_{k,k-1}$ to $\ell_{k,k-1}'=\ell_{k,k-1}-\epsilon$ if $k=m$; and modify $\ell_{k,k-1}$  so that the constraints in \eqref{relax} remain satisfied (assuming $v_k'= v_k$) after $S_{k,k-1}$ is changed if $k<m$.

According to the construction in Algorithm \ref{algorithm tree}, $w'(\epsilon)$ may only violate \eqref{BFM v} and \eqref{relax} out of all the constraints in the SOCP relaxation.

\subsection*{Step (S2')}
It suffices to prove that under the conditions given in Lemma \ref{lemma: exact}, for any optimal solution $w$ of the SOCP relaxation that violates \eqref{BFM 4}, there exists $\epsilon>0$ such that $w'(\epsilon)$ is feasible for the SOCP relaxation and has a smaller objective value than $w$. The following lemma gives a sufficient condition under which $w'(\epsilon)$ is indeed feasible and ``better''. To state the lemma, let $S_{0,-1}\eqdef -s_0$ denote the power that the substation injects to the main grid.

	\begin{lemma}\label{lemma: observation tree}
	Given a feasible point $w=(s,S,\ell,v,s_0)$ of the SOCP relaxation that violates \eqref{BFM 4}, let $m$ be defined as \eqref{m tree}, $\epsilon\in(0,\ell_{m,m-1}-|S_{m,m-1}|^2/v_m]$, and $w'(\epsilon)=(s',S',\ell',v',s_0')$ be the output of Algorithm \ref{algorithm tree}.
    \begin{itemize}
    \item If $\hat{v}_i(s)\leq \overline{v}_i$ for $i\in \hN^+$, and
	\begin{equation}\label{observation tree}
	S_{i,i-1}' > S_{i,i-1}, \qquad 0\leq i\leq m-1,
	\end{equation}
then $w'(\epsilon)$ is feasible for the SOCP relaxation.
    \item If further, $f_0$ is strictly increasing, then $w'(\epsilon)$ has a smaller objective value than $w$, i.e.,
        $$\sum_{i=0}^n f_i(\re(s_i')) < \sum_{i=0}^nf_i(\re(s_i)).$$
    \end{itemize}
	\end{lemma}	
	
	\begin{proof}
	As having been discussed in the end of Step (S1'), to check that $w'(\epsilon)$ is feasible for the SOCP relaxation, it suffices to show that $w'(\epsilon)$ satisfies \eqref{BFM v} and \eqref{relax}, i.e., $v_i'\geq\underline{v}_i$ for $i\in\hN^+$, $\ell_{ij}'\geq |S_{ij}'|^2/v_i'$ for $i\rightarrow j$, and $v_i'\leq\overline{v}_i$ for $i\in\hN^+$.
	
	First show that $v_i'\geq\underline{v}_i$ for $i\in\hN^+$.  If \eqref{observation tree} holds, then $\Delta S_{i,i-1}>0$ for $i=1,\ldots,m-1$ and $\Delta S_{ij}=0$ for $(i,j)\notin \hP_{m-1}$. It follows from \eqref{BFM 1} and \eqref{BFM 2} that
	\begin{eqnarray*}
	\Delta v_i - \Delta v_j = \re ( z_{ij}^* \Delta S_{ij}) + \re \left( z_{ij}^* \sum_{k:\,j\rightarrow k}\Delta S_{jk} \right) \geq 0
	\end{eqnarray*}
for $i\rightarrow j$. Sum up the inequalities over $\hP_i$ to obtain $\Delta v_i \geq \Delta v_0=0$, which implies $v_i' \geq v_i\geq\underline{v}_i$ for $i\in \hN^+$.

Next show that $\ell_{ij}'\geq |S_{ij}'|^2/v_i'$ for $i\rightarrow j$. For $(i,j)\notin\hP_m$, one has $\ell_{ij}'=\ell_{ij}\geq|S_{ij}|^2/v_i=|S_{ij}'|^2/v_i\geq|S_{ij}'|^2/v_i'$; for $(i,j)=(m,m-1)$, one has $\ell_{ij}'=\ell_{ij}-\epsilon\geq|S_{ij}|^2/v_i=|S_{ij}'|^2/v_i\geq|S_{ij}'|^2/v_i'$; for $(i,j)\in\hP_{m-1}$, one has
	\begin{eqnarray*}
	\ell_{ij}' &\!\!\!=& \!\!\!\frac{\max\{P_{ij}'^2,P_{ij}^2\}+\max\{Q_{ij}'^2,Q_{ij}^2\}}{v_i}\\
	&\!\!\!\geq& \!\!\!\frac{P_{ij}'^2+Q_{ij}'^2}{v_i} ~=~ \frac{|S_i'|^2}{v_i} ~\geq~ \frac{|S_i'|^2}{v_i'}.
	\end{eqnarray*}
Hence, $\ell_{ij}'\geq |S_{ij}'|^2/v_i'$ for $i\rightarrow j$.

Finally show that $v_i'\leq \overline{v}_i$ for $i\in\hN^+$. Since $\ell_{ij}'\geq0$ for $i\rightarrow j$, it follows from \eqref{BFM 2} that
	\begin{eqnarray*}
	S_{jk}' = s_j' + \sum_{i:\, i\rightarrow j} \left( S_{ij}' - z_{ij}\ell_{ij}' \right) \leq s_j' + \sum_{i:\, i\rightarrow j} S_{ij}'
	\end{eqnarray*}
for $j\rightarrow k$. Since the network is a tree, the inequality leads to
	\begin{eqnarray*}
	S_{jk}' &\!\!\!\leq& \!\!\!\sum_{i:\,j\in\hP_i} s_i' ~=~ \hat{S}_{jk}(s')
	\end{eqnarray*}
for $j\rightarrow k$. Therefore, it follows from \eqref{BFM 1} that
 	\begin{eqnarray*}
	v_i'-v_j' &=& 2\re(z_{ij}^*S_{ij}') -|z_{ij}|^2\ell_{ij}' \\
	&\leq& 2\re(z_{ij}^*S_{ij}') \\
	&\leq& 2\re(z_{ij}^*\hat{S}_{ij}(s'))
	\end{eqnarray*}
for $i\rightarrow j$. Now, sum up the inequality over $\hP_i$ to obtain
	\begin{eqnarray*}
	v_i'-v_0' \leq 2\sum_{(k,l)\in\hP_i}\re(z_{kl}^*\hat{S}_{kl}(s'))
	\end{eqnarray*}
for $i\in \hN^+$, which implies
    $$v_i'\leq v_0 + 2\sum_{(k,l)\in\hP_i}\re(z_{kl}^*\hat{S}_{kl}(s')) =\hat{v}_i(s') =\hat{v}_i(s) \leq \overline{v}_i$$
for $i\in \hN^+$.

To this end, it has been proved that if $\hat{v}_i(s)\leq\overline{v}_i$ for $i\in\hN^+$ and \eqref{observation tree} holds, then $w'(\epsilon)$ is feasible for the SOCP relaxation.

If further, $f_0$ is strictly increasing, then one has
	\begin{eqnarray*}
	&& \sum_{i=0}^n f_i(\re(s_i')) - \sum_{i=0}^n f_i(\re(s_i)) \\
    	&=& f_0(\re(s_0')) - f_0(\re(s_0)) \\
    	&=& f_0(-\re(S_{0,-1}')) - f_0(-\re(S_{0,-1}))< 0,
	\end{eqnarray*}
i.e., $w'(\epsilon)$ has a smaller objective value than $w$. This completes the proof of Lemma \ref{lemma: observation tree}.
	\end{proof}

\subsection*{Step (S3')}
We transform \eqref{observation tree}, which depends on both $w$ and $w'$, to \eqref{key step tree}, that only depends on $w$, in this step. The idea is to approximate $w'(\epsilon)$ by its Taylor expansion near $\epsilon=0$.

First compute the Taylor expansion of $w'(\epsilon)$. It follows from
	\begin{eqnarray*}
	\Delta S_{m-1,m-2} &=& \Delta S_{m,m-1}-z_{m,m-1}\Delta \ell_{m,m-1} \\
	&=& z_{m,m-1}\epsilon
	\end{eqnarray*}
that
 	\begin{equation}\label{initial tree}
    \begin{pmatrix} \Delta P_{m-1,m-2}\\ \Delta Q_{m-1,m-2} \end{pmatrix}
	= \begin{pmatrix} r_{m,m-1}\\ x_{m,m-1} \end{pmatrix} \epsilon > 0.
    \end{equation}
For any $t\in\{1,\ldots,m-1\}$, if
	\begin{equation}\label{recursive condition tree}
    \begin{pmatrix} \Delta P_{t,t-1}\\ \Delta Q_{t,t-1} \end{pmatrix} = B\epsilon + O(\epsilon^2)
    \end{equation}
for some $B>0$, then $(\Delta P_{t,t-1},\Delta Q_{t,t-1})^T>0$ for sufficiently small $\epsilon$. It follows that
	\begin{eqnarray*}
	\Delta \ell_{t,t-1} &\!\!\!=& \!\!\!\frac{\max\{P_{t,t-1}'^2,P_{t,t-1}^2\}+\max\{Q_{t,t-1}'^2,Q_{t,t-1}^2\}}{v_t} \\
	&& \!\!\! \qquad -\frac{P_{t,t-1}^2+Q_{t,t-1}^2}{v_t}\\
	&\!\!\!=& \!\!\!\frac{\max\{P_{t,t-1}'^2-P_{t,t-1}^2,0\}+\max\{Q_{t,t-1}'^2-Q_{t,t-1}^2,0\}}{v_t}\\
	&\!\!\!=& \!\!\!\frac{\max\{2P_{t,t-1}\Delta P_{t,t-1}+O(\epsilon^2),0\}}{v_t}  \\
	&& \!\!\!\qquad +\frac{\max\{2Q_{t,t-1}\Delta Q_{t,t-1}+O(\epsilon^2),0\}}{v_t}\\
	&\!\!\!=& \!\!\!\frac{2P_{t,t-1}^+\Delta P_{t,t-1}}{v_t} +\frac{2Q_{t,t-1}^+\Delta Q_{t,t-1}}{v_t}+O(\epsilon^2)
	\end{eqnarray*}
and
	\begin{eqnarray}
	&& \!\!\!\!\!\!\!\!\!\!\!\!\!\!\! \begin{pmatrix} \Delta P_{t-1,t-2}\\ \Delta Q_{t-1,t-2} \end{pmatrix} \,=\,
	\begin{pmatrix} \Delta P_{t,t-1}\\ \Delta Q_{t,t-1} \end{pmatrix}
	-\begin{pmatrix} r_{t,t-1}\\ x_{t,t-1} \end{pmatrix}\Delta \ell_{t,t-1} \nonumber\\
	&\!\!\!\!\!=& \!\!\!\!\! \begin{pmatrix} 1-\frac{2r_{t,t-1}P_{t,t-1}^+}{v_t} & -\frac{2r_{t,t-1}Q_{t,t-1}^+}{v_t}\\ -\frac{2x_{t,t-1}P_{t,t-1}^+}{v_t} & 1-\frac{2x_{t,t-1}Q_{t,t-1}^+}{v_t} \end{pmatrix}
	\begin{pmatrix} \Delta P_{t,t-1}\\ \Delta Q_{t,t-1} \end{pmatrix} \nonumber\\
	&& \!\!\!\!\! \qquad +O(\epsilon^2).\label{recursive tree}
	\end{eqnarray}
With this recursive relation and the initial value in \eqref{initial tree}, the Taylor expansion of $(\Delta P_{t,t-1},\Delta Q_{t,t-1})^T$ near $\epsilon=0$ can be computed for $0\leq t \leq m-1$, as given in Lemma \ref{lemma: Taylor}.
	
To state the lemma, given any feasible point $w=(s,S=P+\ii Q,\ell,v,s_0)$ of the SOCP relaxation, define
	\begin{eqnarray*}
	c_{ij}(w) &:=& 1-2r_{ij}P_{ij}^+/v_i, \\
	d_{ij}(w) &:=& 2r_{ij}Q_{ij}^+/v_i, \\
	e_{ij}(w) &:=& 2x_{ij}P_{ij}^+/v_i, \\
	f_{ij}(w) &:=& 1-2x_{ij}Q_{ij}^+/v_i,
	\end{eqnarray*}
and
	\begin{eqnarray*}
	A_{ij}(w)=\begin{pmatrix} c_{ij}(w) & -d_{ij}(w)\\ -e_{ij}(w) & f_{ij}(w) \end{pmatrix}
	\end{eqnarray*}
for $i\rightarrow j$. Also define $u_{ij}:=(r_{ij},x_{ij})^T$ for $i\rightarrow j$.

    \begin{lemma}\label{lemma: Taylor tree}
    Given a feasible point $w=(s,S,\ell,v,s_0)$ of the SOCP relaxation that violates \eqref{BFM 4}, let $m$ be defined as \eqref{m tree}, $\epsilon\in(0,\ell_{m,m-1}-|S_{m,m-1}|^2/v_m]$, and $w'(\epsilon)=(s',S',\ell',v',s_0')$ be the output of Algorithm \ref{algorithm tree}. If
	\begin{equation*}
	A_{t,t-1}(w)\cdots A_{m-1,m-2}(w)u_{m,m-1} > 0
	\end{equation*}
for $t=1,\ldots,m$, then
	\begin{equation}\label{Taylor tree}
	\begin{pmatrix} \Delta P_{t-1,t-2}\\ \Delta Q_{t-1,t-2} \end{pmatrix} = A_{t,t-1}\cdots A_{m-1,m-2}u_{m,m-1}\epsilon+O(\epsilon^2)
	\end{equation}
for $t=1,\ldots,m$.
    \end{lemma}
    \begin{proof}
    We prove that \eqref{Taylor tree} holds for $t=m,\ldots,1$ by mathematical induction on $t$.
    \begin{itemize}
    \item[i)] When $t=m$, equation \eqref{Taylor tree} holds according to \eqref{initial tree}.
    \item[ii)] Assume that \eqref{Taylor tree} holds for $t=k$ ($2\leq k\leq m$), i.e.,
    $$\!\!\!\!\! \begin{pmatrix} \Delta P_{k-1,k-2}\\ \Delta Q_{k-1,k-2} \end{pmatrix} = A_{k,k-1}\cdots A_{m-1,m-2}u_{m,m-1}\epsilon+O(\epsilon^2).$$
    Since $A_{k,k-1}\cdots A_{m-1,m-2}u_{m,m-1}>0$, equation \eqref{recursive condition tree} holds for $t=k-1$. It follows that when $t=k-1$, one has
        \begin{eqnarray*}
	    && \!\!\!\!\!\!\!\!\!\! \begin{pmatrix} \Delta P_{t-1,t-2}\\ \Delta Q_{t-1,t-2} \end{pmatrix}
	    \,=\,  A_{t,t-1} \begin{pmatrix} \Delta P_{t,t-1}\\ \Delta Q_{t,t-1} \end{pmatrix} +O(\epsilon^2) \\
        &\!\!\!\!\!\!=& \!\!\! A_{t,t-1}  A_{k,k-1}\cdots A_{m-1,m-2}u_{m,m-1}\epsilon +O(\epsilon^2),
	    \end{eqnarray*}
i.e., equation \eqref{Taylor tree} holds for $t=k-1$.
    \end{itemize}
According to (i) and (ii), equation \eqref{Taylor tree} holds for $t=m,\ldots,1$. This completes the proof of Lemma \ref{lemma: Taylor tree}.
    \end{proof}

The condition in \eqref{observation tree} can be simplified to a new condition that only depends on $w$, as given in the following corollary, using the Taylor expansion of $(\Delta P,\Delta Q)$ given in Lemma \ref{lemma: Taylor tree}. To state the corollary, for $i\in\hN^+$, let $n_i+1$ denote the number of buses on path $\hP_i$, and denote
    $$\hP_i=\{(i_{n_i}, i_{n_i-1}), \ldots, (i_1,i_0)\}.$$
Then $i_{n_i}=i$ and $i_0=0$. Define
    $$A_{i,t} = \begin{pmatrix} 1-\frac{2r_{i_ti_{t-1}}P_{i_ti_{t-1}}^+}{v_{i_t}} & -\frac{2r_{i_ti_{t-1}}Q_{i_ti_{t-1}}^+}{v_{i_t}}\\ -\frac{2x_{i_ti_{t-1}}P_{i_ti_{t-1}}^+}{v_{i_t}} & 1-\frac{2x_{i_ti_{t-1}}Q_{i_ti_{t-1}}^+}{v_{i_t}} \end{pmatrix}$$
for $i\in\hN^+$ and $1\leq t \leq n_i$.

	\begin{corollary}\label{lemma: key lemma tree}
Let $w$ be a feasible point of the SOCP relaxation that violates \eqref{BFM 4}. If
	\begin{equation}\label{key step tree}
	A_{i_t,i_{t-1}}(w)\cdots A_{i_{n_i-1},i_{n_i-2}}(w) u_{i_{n_i},i_{n_i-1}}>0
	\end{equation}
for $i\in\hN^+$ and $1\leq t\leq n_i$, then the output $w'(\epsilon)$ of Algorithm \ref{algorithm tree} satisfies \eqref{observation tree} for sufficiently small $\epsilon$.
	\end{corollary}
    \begin{proof}
    Let $m$ be defined as \eqref{m tree}, then it follows from \eqref{path} that $n_m=m$ and $m_t=t$ for $0\leq t\leq m$. Substitute $i=m$ in \eqref{key step tree} to obtain
        $$A_{t,t-1}(w)\cdots A_{m-1,m-2}(w)u_{m,m-1}>0, \quad 1\leq t \leq  m.$$
    Then, it follows from Lemma \ref{lemma: Taylor tree} that
        $$\begin{pmatrix} \Delta P_{t-1,t-2}\\ \Delta Q_{t-1,t-2} \end{pmatrix} = A_{t,t-1}\cdots A_{m-1,m-2}u_{m,m-1}\epsilon+O(\epsilon^2)$$
    for $t=1,\ldots,m$. Hence, $\Delta S_{t,t-1}>0$ for $t=0,\ldots,m-1$, for sufficiently small $\epsilon$, i.e., $w'(\epsilon)$ satisfies \eqref{observation tree} for sufficiently small $\epsilon$.
    \end{proof}
    

	\begin{corollary}\label{corollary: key tree}
    Assume that $f_0$ is strictly increasing; every optimal solution $w=(s,S,\ell,v,s_0)$ of the SOCP relaxation satisfies $\hat{v}_i(s)\leq\overline{v}_i$ for $i\in\hN^+$ and \eqref{key step tree}. Then the SOCP relaxation is exact.
	\end{corollary}

    \begin{proof}
    Assume the SOCP relaxation is not exact, then one can derive a contradiction as follows.

    Since the SOCP relaxation is not exact, there exists an optimal solution $w=(s,S,\ell,v,s_0)$ of the SOCP relaxation that violates \eqref{BFM 4}. Since $w$ satisfies \eqref{key step tree}, one can pick a sufficiently small $\epsilon>0$ such that $w'(\epsilon)$ satisfies \eqref{observation tree} according to Corollary \ref{lemma: key lemma tree}. Further, the point $w'(\epsilon)$ is feasible for the SOCP relaxation and has a smaller objective value than $w$, according to Lemma \ref{lemma: observation tree}. This contradicts with $w$ being optimal for the SOCP relaxation.

    Hence, the SOCP relaxation is exact.
    \end{proof}

\subsection*{Step (S4')}
To complete the proof of Lemma \ref{lemma: exact}, it suffices to prove the following lemma.
	\begin{lemma}\label{lemma: main condition tree}
	If there exists $\overline{p}_i$ and $\overline{q}_i$ such that \eqref{relax constraints} holds for $i\in\hN^+$, and C1 holds, then \eqref{key step tree} holds for all feasible points of the SOCP relaxation.
	\end{lemma}

    \begin{proof}
    Assume that there exists $\overline{p}_i$ and $\overline{q}_i$ such that \eqref{relax constraints} holds for $i\in\hN^+$; and that C1 holds. According to Lemma \ref{lemma: matrix multiplication}, it suffices to prove that for any feasible point $w=(s=p+\ii q,S=P+\ii Q,\ell,v)$ of the SOCP relaxation, the following inequalities hold. To state the inequalities, let
    	$$\hL:=\{i\in\hN~|~\nexists h \text{ such that }h\rightarrow i\}$$
denote the collection of leaf buses.
    \begin{itemize}
    \item[i)] $0<c_{ij}(w)\leq 1, d_{ij}(w)\geq 0, e_{ij}(w)\geq 0, 0<f_{ij}(w)\leq1$ for $i\notin\hL$, $i\rightarrow j$;
    \item[ii)]
    \[\begin{pmatrix}
	\displaystyle \prod_{(k,l)\in\hP_j} c_{kl}(w) & \displaystyle  -\sum_{(k,l)\in\hP_j}d_{kl}(w)\\
	\displaystyle  -\sum_{(k,l)\in\hP_j} e_{kl}(w) & \displaystyle  \prod_{(k,l)\in\hP_j}f_{kl}(w)
	\end{pmatrix}
	\begin{pmatrix}
	r_{ij} \\ x_{ij}
	\end{pmatrix}
	> 0\]
for $i\rightarrow j$.
    \end{itemize}

	First check (i). It follows from $a_j^1r_{ij}>a_j^2x_{ij}\geq0$ (see C1) for $i\rightarrow j$ that $a_j^1>0$ for $j\notin\hL$. Now fix any $i\notin\hL$. Since
    $$0<a_j^1=\prod_{(k,l)\in\hP_j}\left( 1-\frac{2r_{kl}\hat{P}_{kl}^+(\overline{p})}{\underline{v}_k} \right)$$
    for $j\in\hP_i$, one has
    $$1-2r_{kl}\hat{P}_{kl}^+(\bar{p})/\underline{v}_k>0$$
for $(k,l)\in\hP_i$. Since $\ell_{jk}\geq0$ for $j\rightarrow k$, equation \eqref{BFM 2} implies
	\begin{eqnarray*}
	S_{kl} = s_k + \sum_{j:\, j\rightarrow k} \left( S_{jk} - z_{jk}\ell_{jk} \right) \leq s_k + \sum_{j:\, j\rightarrow k} S_{jk}
	\end{eqnarray*}
for $k\rightarrow l$. Since the network is a tree, the inequality leads to
	\begin{eqnarray*}
	S_{kl} \leq \!\!\!\sum_{j:\,k\in\hP_j} s_j = \hat{S}_{kl}(s)
	\end{eqnarray*}
for $k\rightarrow l$. It follows that $P_{kl}\leq \hat{P}_{kl}({p})\leq \hat{P}_{kl}(\overline{p})$ and therefore $P_{kl}^+\leq \hat{P}_{kl}^+(p)\leq \hat{P}_{kl}^+(\overline{p})$ for $k\rightarrow l$. It follows that
	\begin{eqnarray*}
	&& c_{kl}(w) = 1-\frac{2r_{kl}P_{kl}^+}{v_k}\in \left[1-\frac{2r_{kl}\hat{P}_{kl}^+(\bar{p})}{\underline{v}_k},1\right]\subseteq \left(0,1\right],\\
	&& e_{kl}(w) = \frac{2x_{kl}P_{kl}^+}{v_k}\in \left[0,\frac{2x_{kl}\hat{P}_{kl}^+(\bar{p})}{\underline{v}_k}\right] \subseteq \left[0,\infty\right)
	\end{eqnarray*}
for $(k,l)\in\hP_i$. Similarly,
	\begin{eqnarray*}
	&& d_{kl}(w) \in \left[ 0, \frac{2r_{kl}\hat{Q}_{kl}^+(\bar{q})}{\underline{v}_k}\right]\subseteq \left[0,\infty\right),\\
	&& f_{kl}(w) \in \left[1-\frac{2x_{kl}\hat{Q}_{kl}^+(\bar{q})}{\underline{v}_k}, 1\right]\subseteq \left(0,1\right]
	\end{eqnarray*}
for $(k,l)\in\hP_i$. Hence, the inequalities in (i) hold.

Next check (ii). For any $i\rightarrow j$, one has
    \begin{eqnarray*}
	\prod_{(k,l)\in\hP_j} c_{kl}(w)r_{ij}
    &=& \prod_{(k,l)\in\hP_j} \left(1-\frac{2r_{kl}P_{kl}^+}{v_k}\right) r_{ij} \\
    &\geq& \prod_{(k,l)\in\hP_j} \left(1-\frac{2r_{kl}\hat{P}_{kl}^+(\bar{p})}{\underline{v}_k}\right) r_{ij} \\
    &=& a_j^1r_{ij} \\
    &>& a_j^2x_{ij} \qquad \text{(Condition C1)}\\
    &=& \sum_{(k,l)\in\hP_j}\frac{2r_{kl}\hat{Q}_{kl}^+(\overline{q})}{\underline{v}_k}x_{ij} \\
    &\geq& \sum_{(k,l)\in\hP_j}\frac{2r_{kl}Q_{kl}^+}{\underline{v}_k}x_{ij} \\
    &=& \sum_{(k,l)\in\hP_j}d_{kl}(w)x_{ij}.
    \end{eqnarray*}
Similarly,
    \begin{eqnarray*}
	\sum_{(k,l)\in\hP_j} e_{kl}(w)r_{ij} < \prod_{(k,l)\in\hP_j} f_{kl}(w)x_{ij}
	\end{eqnarray*}
for $i\rightarrow j$. Hence, the inequalities in (ii) hold.

To this end, we have proved that the inequalities in (i) and (ii) hold. Apply Lemma \ref{lemma: matrix multiplication} to obtain
    \begin{eqnarray*}
    &&\!\!\!\!\!\!\!\!\!\! \begin{pmatrix}
	c_{i_t,i_{t-1}} & -d_{i_t,i_{t-1}}\\
	-e_{i_t,i_{t-1}} & f_{i_t,i_{t-1}}
	\end{pmatrix}
    \cdots \\
    && \begin{pmatrix}
	c_{i_{n_i-1},i_{n_i-2}} & -d_{i_{n_i-1},i_{n_i-2}}\\
	-e_{i_{n_i-1},i_{n_i-2}} & f_{i_{n_i-1},i_{n_i-2}}
	\end{pmatrix}
	\begin{pmatrix}
	r_{i_{n_i},i_{n_i-1}} \\ x_{i_{n_i},i_{n_i-1}}
	\end{pmatrix}
	> 0
	\end{eqnarray*}
for $i\in\hN^+$ and $1\leq t < n_i$, i.e.,
    $$A_{i_t,i_{t-1}}(w)\cdots A_{i_{n_i-1},i_{n_i-2}}(w) u_{i_{n_i},i_{n_i-1}}>0$$
for $i\in\hN^+$ and $1\leq t < n_i$. When $t=n_i$, one has
    $$A_{i_t,i_{t-1}}\cdots A_{i_{n_i-1},i_{n_i-2}} u_{i_{n_i},i_{n_i-1}}= u_{i_{n_i},i_{n_i-1}}>0.$$
Hence, $A_{i_t,i_{t-1}}\cdots A_{i_{n_i-1},i_{n_i-2}} u_{i_{n_i},i_{n_i-1}}>0$ for $i\in\hN^+$, $1\leq t\leq n_i$. This completes the proof of Lemma \ref{lemma: main condition tree}.
\end{proof}

\noindent{\it Proof of Lemma \ref{lemma: exact}. }
If the conditions in Lemma \ref{lemma: exact} hold, then the conditions in Lemma \ref{lemma: main condition tree} hold. It follows from Lemma \ref{lemma: main condition tree} that \eqref{key step tree} holds for all optimal solution of the SOCP relaxation.

Then, the conditions in Corollary \ref{corollary: key tree} hold, and it follows that the SOCP relaxation is exact. This completes the proof of Lemma \ref{lemma: exact}. $\hfill\Box$

\section{Supplements for the main sections}
\subsection{Proof of Theorem \ref{thm: equivalence}}\label{app: equivalence}
It is straightforward that for any $x=(s,V,s_0)\in\mathcal{F}_\text{OPF}$, the point $\phi(x)=(s,W,s_0)$ where $W$ is defined according to \eqref{phi} is feasible for OPF' and has the same objective value as $x$. It remains to prove that the map $\phi$ is bijective.

We first show that the map $\phi$ is injective. For $x=(s,V,s_0)$ and $x'=(s',V',s_0')$, if $\phi(x)=\phi(x')$, then $V_iV_j^*=V_i'V_j'^*$ for $i\sim j$. Hence, $V_i=V_i'$ implies $V_j=V_j'$ if $i\sim j$. Then, since $V_0=V_0'$ and the network is connected, $V=V'$, which further implies $x=x'$.

At last, we prove that $\phi$ is surjective. It suffices to show that for any $(s,W,s_0)\in\hF_\text{OPF'}$, there exists $V$ such that $V_0$ equals the fixed constant and $W_{ij}=V_iV_j^*$ for $i=j$ and $i\sim j$. Such $V$ is given below:
	$$V_i=\sqrt{W_{ii}}\exp\left(\ii \left(\angle V_0+\sum_{(j,k)\in\hP_i}\angle W_{jk}\right)\right), \qquad i\in \hN.$$
It is not difficult to verify that such $V$ satisfies $\phi(s,V,s_0)=(s,W,s_0)$, which completes the proof of Theorem \ref{thm: equivalence}.

\subsection{The SOCP relaxation is not always exact.}\label{app: non exact}
The SOCP relaxation is not always exact, and a two-bus example where the SOCP relaxation is not exact is illustrated in Fig. \ref{fig: example}. Bus 0 is the substation and has fixed voltage magnitude $|V_0|=1$. The branch bus has distributed generator generating 1 per unit real power, of which $\re(s)\in[0,1]$ is injected into the grid and the rest $1-\re(s)$ is curtailed. The reactive power injection is 0. Line admittance is $y=2-4\ii$. The objective is to minimize the sum of power loss $2(W_{00}-W_{01}-W_{10}+W_{11})$ and curtailment $1-\re(s)$.
	\begin{figure}[!htbp]
     	\centering
     	\includegraphics[scale=0.5]{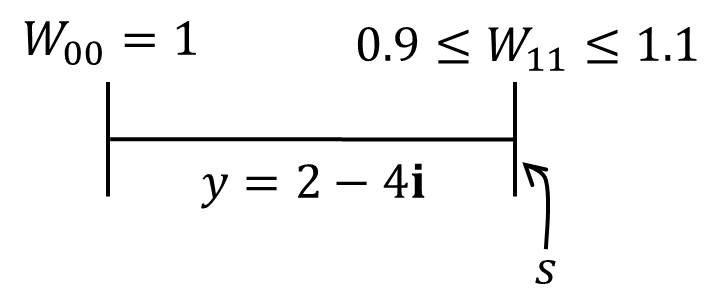}
      	\caption{A two-bus network example where the SOCP relaxation is not exact.}
      	\label{fig: example}
	\end{figure}

In this example, the SOCP relaxation is
	\begin{eqnarray*}
	\underset{s,W}{\min} && 2(W_{00}-W_{01}-W_{10}+W_{11})+(1-\re(s))\\
	\mathrm{s.t.} && s=(2+4\ii)(W_{11}-W_{10}),\\
	&& 0\leq\re(s)\leq 1,~ \im(s)=0,\\
	&& W_{00}=1, ~0.9\leq W_{11}\leq 1.1,\\
	&& W\{0,1\}\succeq0,
	\end{eqnarray*}
whose solution is
	$$W\{0,1\}=\begin{pmatrix}
	1 & 1-0.2\ii \\
	1+0.2\ii & 1.1
	\end{pmatrix}$$
and not rank-one. Therefore, the SOCP relaxation is not exact in this example.

\subsection{Proof of Lemma \ref{lemma: v}}\label{app: lemma v}
Let $(s,S,W,s_0)$ be an arbitrary point that satisfies \eqref{BIM 1}, \eqref{sdp}, and \eqref{PF S}. It can be verified that \eqref{sdp} implies $W_{ii}-W_{ij}-W_{ji}+W_{jj}\geq0$ for $i\rightarrow j$. Then it follows from \eqref{BIM 1} that for $i\rightarrow j$,
	\begin{eqnarray*}
	&& S_{ij}=(W_{ii}-W_{ij})y_{ij}^* \\
	&=& s_i - \sum_{h:\,h\rightarrow i} (W_{ii}-W_{ih})y_{ih}^*\\
	&=& s_i + \sum_{h:\,h\rightarrow i} S_{hi} -(W_{ii}-W_{ih}-W_{hi}+W_{hh})y_{ih}^*\\
	&\leq& s_i + \sum_{h:\,h\rightarrow i} S_{hi}.
	\end{eqnarray*}
On the other hand, $\hat{S}_{ij}(s)$ is the solution to
	\begin{equation*}
	S_{ij} = s_i + \sum_{h:\,h\rightarrow i} S_{hi}, \qquad i\rightarrow j.
	\end{equation*}
Hence, by induction from the leaf edges, one can show that
	\begin{equation*}
	S_{ij} \leq \hat{S}_{ij}(s), \qquad i\rightarrow j.
	\end{equation*}

At last,
	\begin{eqnarray*}
	&& 0 \leq W_{ii}-W_{ij}-W_{ji}+W_{jj}\\
	&=& 2\re(W_{ii}-W_{ij}) + W_{jj}-W_{ii}\\
	&=& 2\re(z_{ij}^*S_{ij}) + W_{jj}-W_{ii}\\
	&\leq& 2\re(z_{ij}^*\hat{S}_{ij}(s)) + W_{jj}-W_{ii}
	\end{eqnarray*}
for $i\rightarrow j$. Now, sum up the inequality over $\hP_i$ to obtain
	\begin{eqnarray*}
	0 \leq 2\sum_{(j,k)\in\hP_i}\re(z_{jk}^*\hat{S}_{jk}(s)) +W_{00}-W_{ii}
	\end{eqnarray*}
for $i\in \hN^+$, which implies $W_{ii}\leq \hat{W}_{ii}(s)$ for $i\in \hN^+$.

\subsection{Proof of Theorem \ref{thm: unique}}\label{app: thm unique}
Suppose that SOCP-m is convex, exact, and has at least one solution. Let $\tilde{w}=(\tilde{s},\tilde{W},\tilde{s}_0)$ and $\hat{w}=(\hat{s},\hat{W},\hat{s}_0)$ denote two arbitrary solutions to SOCP-m. It suffices to prove that $\tilde{w}=\hat{w}$.

It follows from SOCP-m being exact that $\tilde{W}_{ii}\tilde{W}_{jj}=\tilde{W}_{ij}\tilde{W}_{ji}$ and $\hat{W}_{ii}\hat{W}_{jj}=\hat{W}_{ij}\hat{W}_{ji}$ for $i\rightarrow j$. For any $\theta\in(0,1)$, define $w(\theta)\eqdef\theta \tilde{w}+(1-\theta)\hat{w}$. It follows from SOCP-m being convex that $w(\theta)$ is optimal for SOCP-m. Therefore, $W_{ii}(\theta)W_{jj}(\theta)=W_{ij}(\theta)W_{ji}(\theta)$ for $i \rightarrow j$. Substitute $W_{ij}(\theta) = \theta \tilde{W}_{ij} + (1-\theta) \hat{W}_{ij}$ for $i=j$ or $i\sim j$ to obtain $\tilde{W}_{ii}\hat{W}_{jj} + \hat{W}_{ii}\tilde{W}_{jj} = \tilde{W}_{ij}\hat{W}_{ji}+\hat{W}_{ij}\tilde{W}_{ji}$ for $i\rightarrow j$.

It follows that
	\begin{eqnarray*}
	&& \tilde{W}_{ii}\hat{W}_{jj} + \hat{W}_{ii}\tilde{W}_{jj}= \tilde{W}_{ij}\hat{W}_{ji}+\hat{W}_{ij}\tilde{W}_{ji}\\
	&\leq& 2|\tilde{W}_{ij}||\hat{W}_{ij}|= 2\sqrt{\tilde{W}_{ii}\tilde{W}_{jj}\hat{W}_{ii}\hat{W}_{jj}}\\
	&\leq& \tilde{W}_{ii}\hat{W}_{jj} + \hat{W}_{ii}\tilde{W}_{jj}
	\end{eqnarray*}
for $i\rightarrow j$. The two inequalities must both attain equalities. The first inequality attaining equality implies that $\angle\tilde{W}_{ij}=\angle\hat{W}_{ij}$. The second inequality attaining equality implies that $\hat{W}_{ii}/\tilde{W}_{ii}=\hat{W}_{jj}/\tilde{W}_{jj}$. Define $\eta_i:=\hat{W}_{ii}/\tilde{W}_{ii}$ for $i\in \hN$, then $\eta_0=1$ and $\eta_i=\eta_j$ if $i\rightarrow j$. It follows that $\eta_i=1$ for $i\in \hN$, which implies $\hat{W}_{ii}=\tilde{W}_{ii}$ for $i\in \hN$.

Then, $|\hat{W}_{ij}|=\sqrt{\hat{W}_{ii}\hat{W}_{jj}}=\sqrt{\tilde{W}_{ii}\tilde{W}_{jj}}=|\tilde{W}_{ij}|$ for $i\rightarrow j$. Since it has been proved that $\angle\hat{W}_{ij}=\angle\tilde{W}_{ij}$ for $i\rightarrow j$, we have $\hat{W}_{ij}=\tilde{W}_{ij}$ for $i\rightarrow j$.

To summarize, we have shown that $\hat{W}=\tilde{W}$, from which $\hat{w}=\tilde{w}$ follows. This completes the proof of Theorem \ref{thm: unique}.

\section{Other sufficient conditions}
We use the proof technique of Lemma \ref{lemma: exact}---{\it primal construction method}---to obtain some other sufficient conditions for the exactness of the SOCP relaxation in this appendix, more specifically, the conditions derived in \cite{Masoud11,Bose12,Gan12}.


\subsection{Conditions in \cite{Masoud11,Bose12}}
We first use the primal construction method to derive the conditions given in \cite{Masoud11,Bose12}. To state the results, define
    $$\theta_{ij}:=\angle z_{ij}=\mathrm{arctan}(x_{ij}/r_{ij})$$
as the angle of impedance on line $i\rightarrow j$, and angles
    \begin{eqnarray*}
    \delta_{ij}(\overline{p}_i) &\eqdef& -\theta_{ij},\\
    \delta_{ij}(\overline{q}_i) &\eqdef& -\theta_{ij}+\pi/2,\\
    \delta_{ij}(\underline{p}_i) &\eqdef& -\theta_{ij}+\pi,\\
    \delta_{ij}(\underline{q}_i) &\eqdef& -\theta_{ij}-\pi/2,\\
    \delta_{ij}(\overline{p}_j) &\eqdef& \theta_{ij},\\
    \delta_{ij}(\overline{q}_j) &\eqdef& \theta_{ij}-\pi/2,\\
    \delta_{ij}(\underline{p}_j) &\eqdef& \theta_{ij}-\pi,\\
    \delta_{ij}(\underline{q}_j) &\eqdef& \theta_{ij}+\pi/2
    \end{eqnarray*}
for each bound $\underline{p}_i$, $\overline{p}_i$, $\underline{q}_i$, $\overline{q}_i$, $\underline{p}_j$, $\overline{p}_j$, $\underline{q}_j$, $\overline{q}_j$ related to $i\rightarrow j$.
    \begin{figure}[h]
     	\centering
     	\includegraphics[scale=0.15]{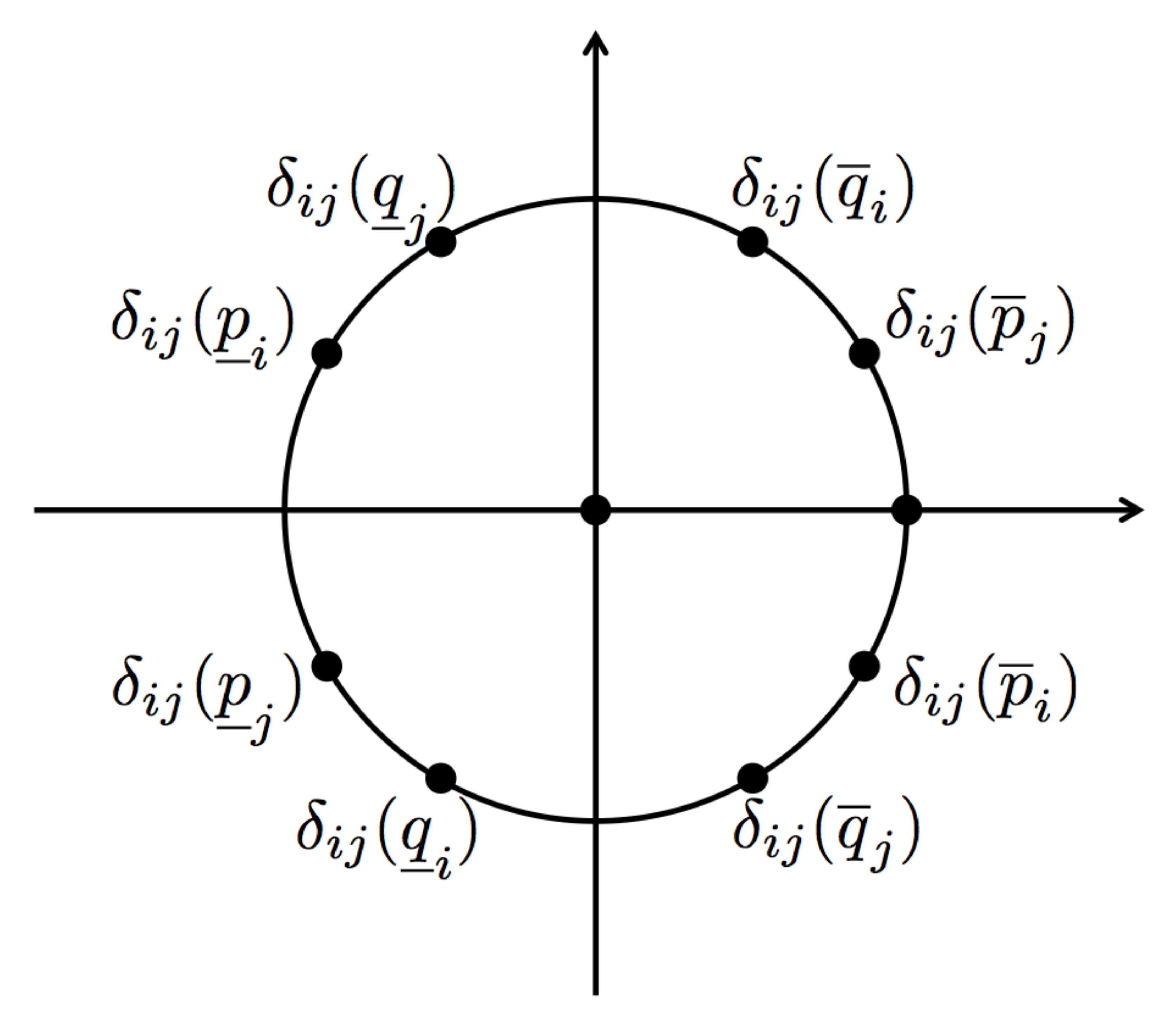}
      	\caption{Illustration of the bounds for line $i\rightarrow j$.}
      	\label{fig: position}
	\end{figure}

Associate a unit circle for every line $i\rightarrow j$, and draw the angles $\delta_{ij}(*)$ corresponding to finite bounds on the unit circle. For example, if all 8 bounds are finite, then the unit circle for line $i\rightarrow j$ is shown in Fig. \ref{fig: position}.

It is proved in Theorem \ref{thm: local exact} that if the bounds for every line satisfies some pattern, more specifically every line is {\it well constrained}, then the SOCP relaxation is exact. The definition of well-constrained is given below.
\begin{definition}
A line $i\rightarrow j$ is {\it well-constrained} if the angles $\delta_{ij}(*)$ corresponding to finite bounds can be placed in a semi-circle that contains angle 0 in its interior.
\end{definition}
To be concrete, we provide three cases where a line $i\rightarrow j$ is well-constrained.
\begin{itemize}
\item[a)] If $\underline{p}_i=\underline{q}_i=\underline{p}_j=\underline{q}_j=-\infty$, then line $i\rightarrow j$ is well-constrained and the semi-circle can be chosen as $[-\pi/2,\pi/2]$. This assumption is called {\it load over-satisfaction} in literature since it requires every bus to be able to draw infinite real and reactive power.
\item[b)] Assume $r_{ij}\geq x_{ij}$. If $\underline{p}_i=\underline{p}_j=-\infty$ and at least one of $\underline{q}_i$ and $\underline{q}_j$ is $-\infty$, then line $i\rightarrow j$ is well-constrained. For example if $\underline{q}_i=-\infty$, then the semi-circle can be chosen as $[\delta_{ij}(\overline{q}_j),\delta_{ij}(\underline{q}_j)]$. This is a slight improvement over load over-satisfaction since it does not require removing all the lower bounds on reactive power injections.
\item[c)] Assume $r_{ij}\leq x_{ij}$. If $\underline{q}_i=\underline{q}_j=-\infty$ and at least one of $\underline{p}_i$ and $\underline{p}_j$ is $-\infty$, then line $i\rightarrow j$ is well-constrained. For example if $\underline{p}_i=-\infty$, then the semi-circle can be chosen as $[\delta_{ij}(\underline{p}_j),\delta_{ij}(\overline{p}_j)]$. This case is similar to case (b).
\end{itemize}

Now we are ready to formally state the sufficient condition for the exactness of the SOCP relaxation.
\begin{theorem}\label{thm: local exact}
Assume $f_i$ is strictly increasing and there exists $\underline{p}_i$, $\overline{p}_i$, $\underline{q}_i$, and $\overline{q}_i$ such that $\mathcal{S}_i=\{s\in\mathbb{C}~|~\underline{p}_i\leq \re(s)\leq \overline{p}_i,~\underline{q}_i\leq \im(s)\leq \overline{q}_i\}$ for $i\in\hN^+$. If every line $i\rightarrow j$ is well-constrained, then the SOCP relaxation is exact.
\end{theorem}

Before proving Theorem \ref{thm: local exact}, we highlight that it is equivalent to the result in \cite{Bose12}, and stronger than the result in \cite{Masoud11}, which says that the SOCP relaxation is exact if there are no lower bounds on power injections.

\begin{corollary}\label{cor: load over-satisfaction}
Assume $f_i$ is strictly increasing and there exists $\underline{p}_i$, $\overline{p}_i$, $\underline{q}_i$, and $\overline{q}_i$ such that $\mathcal{S}_i=\{s\in\mathbb{C}~|~\underline{p}_i\leq \re(s)\leq \overline{p}_i,~\underline{q}_i\leq \im(s)\leq \overline{q}_i\}$ for $i\in\hN^+$. The SOCP relaxation is exact if there are no lower bounds on power injections, i.e., $\underline{p}_i=\underline{q}_i=-\infty$ for $i\in \hN^+$.
\end{corollary}
Corollary \ref{cor: load over-satisfaction} follows from Theorem \ref{thm: local exact} and the fact that no lower bounds on power injections implies every line is well-constrained (case (a)).

Now we give the proof of Theorem \ref{thm: local exact}.
\begin{proof}
It suffices to prove that at the optimal solution $w=(s,W,s_0)$ of the SOCP relaxation, the equality $W_{ii}W_{jj}=|W_{ij}|^2$ holds for every $i\rightarrow j$. We prove this by contradiction: otherwise, $W_{ii}W_{jj} > |W_{ij}|^2$ for some $i\rightarrow j$, and one can construct another feasible point $w'=(s',W',s_0')$ of the SOCP relaxation that has a smaller objective value than $w$, which contradicts with $w$ being optimal for the SOCP relaxation.

The construction of $w'$ is as follows: pick some $\rho>0$ and $\alpha\in(-\pi/2,\pi/2)$, construct $W'$ by
    $$W_{kl}'=\begin{cases}
    W_{kl} & \text{if }\{k,l\}\neq\{i,j\}\\
    W_{ij}+\rho\exp(j\alpha) & \text{if }(k,l)=(i,j)\\
    W_{ji}+\rho\exp(-j\alpha) & \text{if }(k,l)=(j,i),
    \end{cases}$$
and construct $(s',s_0')$ by \eqref{BIM 1}. It can be verified that $w'$ has a smaller objective value than $w$. Hence, it suffices to prove that $w'$ is feasible for the SOCP relaxation for some appropriately chosen $\rho>0$ and $\alpha\in(-\pi/2,\pi/2)$.

The point $w'$ satisfies \eqref{BIM 3} since $W_{ii}'=W_{ii}$ for $i\in \hN^+$. The point $w'$ also satisfies \eqref{sdp} if $\rho$ is sufficiently small. We construct $w'$ in a way that \eqref{BIM 1} is satisfied, then $w'$ may only violate \eqref{BIM 2} out of all the constraints in the SOCP relaxation. In the rest of the proof, we show that if $i\rightarrow j$ is well-constrained, then there exists $\alpha\in(-\pi/2,\pi/2)$ such that \eqref{BIM 2} is satisfied.

Noting that $s_k'=s_k$ if $k\neq i,j$, it suffices to look at $s_i'$ and $s_j'$ to check \eqref{BIM 2}. Define $\Delta s_i\eqdef s_i'-s_i$, then
    \begin{eqnarray*}
    \Delta s_i = -y_{ij}^*\rho\exp(\ii\alpha) = |y_{ij}|\rho\exp\left[\ii(\alpha+\theta_{ij}+\pi)\right]
    \end{eqnarray*}
since $\angle y_{ij}^*=\angle z_{ij}=\theta_{ij}$. To guarantee $\re(s_i')\leq \overline{p}_i$, it suffices if $\re(\Delta s_i)\leq0$, which is equivalent to $-\pi/2-\theta_{ij} \leq \alpha \leq \pi/2-\theta_{ij}$, i.e., $\alpha$ is in the semi-circle centered at $-\theta_{ij}=\delta_{ij}(\overline{p}_i)$.
To summarize,
    $$\re(s_i')\leq \overline{p}_i ~ \Leftarrow ~ \alpha \text{ in the semi-circle centered at } \delta_{ij}(\overline{p}_i).$$
Similarly,
    \begin{eqnarray*}
    \re(s_i')\geq \underline{p}_i & \Leftarrow & \alpha \text{ in the semi-circle centered at } \delta_{ij}(\underline{p}_i),\\
    \im(s_i')\leq \overline{q}_i & \Leftarrow & \alpha \text{ in the semi-circle centered at } \delta_{ij}(\overline{q}_i),\\
    \im(s_i')\geq \underline{q}_i & \Leftarrow & \alpha \text{ in the semi-circle centered at } \delta_{ij}(\underline{q}_i),\\
    \re(s_j')\leq \overline{p}_j & \Leftarrow & \alpha \text{ in the semi-circle centered at } \delta_{ij}(\overline{p}_j),\\
    \re(s_j')\geq \underline{p}_j & \Leftarrow & \alpha \text{ in the semi-circle centered at } \delta_{ij}(\underline{p}_j),\\
    \im(s_j')\leq \overline{q}_j & \Leftarrow & \alpha \text{ in the semi-circle centered at } \delta_{ij}(\overline{q}_j),\\
    \im(s_j')\geq \underline{q}_j & \Leftarrow & \alpha \text{ in the semi-circle centered at } \delta_{ij}(\underline{q}_j).
    \end{eqnarray*}
If $i\rightarrow j$ is well-constrained, then there exists a semicircle with angle 0 in its interior, that contains all the angles $\delta_{ij}(*)$ corresponding to finite bounds. It can be verified that the center of this semi-circle is an appropriate choice of $\alpha$, with which $w'$ satisfies \eqref{BIM 2}. This completes the proof of Theorem \ref{thm: local exact}.
\end{proof}

\subsection{Conditions in \cite{Gan12}}
We use the primal construction method to derive the conditions given in \cite{Gan12}. To state the results, recall the definition of $\hat{S}_{ij}(p+\ii q) = \hat{P}_{ij}(p)+\ii \hat{Q}_{ij}(q)$ in Section \ref{sec: related work} and the definition of $\underline{v}_i$ and $\overline{v}_i$ in Appendix \ref{app: transformation}.
	\begin{lemma}\label{thm: recover}
	Assume $f_0$ is strictly increasing and there exists $\overline{p}_i$ and $\overline{q}_i$ such that \eqref{relax constraints} holds for $i\in\hN^+$. Then the SOCP relaxation is exact if $\overline{v}_i=\infty$ for $i\in \hN^+$ and any one of the following conditions holds.
	\begin{itemize}
	\item[(i)] $\hat{P}_{ij}(\overline{p})\leq0$, $\hat{Q}_{ij}(\overline{q})\leq0$ for all $i\rightarrow j$.
	
	\item[(ii)] $r_{ij}/x_{ij}=r_{jk}/x_{jk}$ for all $i\rightarrow j$, $j\rightarrow k$, and\\ $\underline{v}_i-2r_{ij}\hat{P}_{ij}^+(\overline{p})-2x_{ij}\hat{Q}_{ij}^+(\overline{q})>0$ for all $i\rightarrow j$.
	
	\item[(iii)] $r_{ij}/x_{ij}\geq r_{jk}/x_{jk}$ for all $i\rightarrow j$, $j\rightarrow k$, and\\
	$\hat{P}_{ij}(\overline{p})\leq0$, $\underline{v}_i-2x_{ij}\hat{Q}_{ij}^+(\overline{q})>0$ for all $i\rightarrow j$.
	
	\item[(iv)] $r_{ij}/x_{ij}\leq r_{jk}/x_{jk}$ for all $i\rightarrow j$, $j\rightarrow k$, and\\
	$\hat{Q}_{ij}(\overline{q})\leq0$, $\underline{v}_i-2r_{ij}\hat{P}_{ij}^+(\overline{p})>0$ for all $i\rightarrow j$.
	\end{itemize}
	\end{lemma}

Before proving Lemma \ref{thm: recover}, we highlight that it is identical to Proposition \ref{prop: v}, except for the additional requirements $\underline{v}_i-2r_{ij}\hat{P}_{ij}^+(\overline{p})-2x_{ij}\hat{Q}_{ij}^+(\overline{q})>0$, $\underline{v}_i-2x_{ij}\hat{Q}_{ij}^+(\overline{q})>0$, and $\underline{v}_i-2r_{ij}\hat{P}_{ij}^+(\overline{p})>0$ for $i\rightarrow j$ that are always satisfied in practice.

Condition (i) holds if $\overline{p}_i\leq0$ and $\overline{q}_i\leq0$ for $i\in \hN^+$. This happens if there is neither shunt capacitors nor distributed generators in the network. Condition (ii) holds if the network uses uniform distribution lines. Condition (iii) holds if the distribution lines get thinner as they branch out from the substation, and there are no distributed generators in the network. It is widely satisfied in the current distributed networks which usually do not have distributed generation. Condition (iv) holds if the distribution lines get thicker as they branch out from the substation, and there are no shunt capacitors in the network.

We can remove the restriction of $\overline{v}_i=\infty$ in Lemma \ref{thm: recover} by imposing additional constraints \eqref{modify} on power injections as in SOCP-m, after which the constraints $v_i\leq\overline{v}_i$ are redundant.
	\begin{theorem}\label{cor: recover}
	Assume $f_0$ is strictly increasing and there exists $\overline{p}_i$ and $\overline{q}_i$ such that \eqref{relax constraints} holds for $i\in\hN^+$. Then the SOCP-m relaxation is exact if any one of the following conditions holds.
	\begin{itemize}
	\item[(i)] $\hat{P}_{ij}(\overline{p})\leq0$, $\hat{Q}_{ij}(\overline{q})\leq0$ for all $i\rightarrow j$.
	
	\item[(ii)] $r_{ij}/x_{ij}=r_{jk}/x_{jk}$ for all $i\rightarrow j$, $j\rightarrow k$, and\\ $\underline{v}_i-2r_{ij}\hat{P}_{ij}^+(\overline{p})-2x_{ij}\hat{Q}_{ij}^+(\overline{q})>0$ for all $i\rightarrow j$.
	
	\item[(iii)] $r_{ij}/x_{ij}\leq r_{jk}/x_{jk}$ for all $i\rightarrow j$, $j\rightarrow k$, and\\
	$\hat{P}_{ij}(\overline{p})\leq0$, $\underline{v}_i-2x_{ij}\hat{Q}_{ij}^+(\overline{q})>0$ for all $i\rightarrow j$.
	
	\item[(iv)] $r_{ij}/x_{ij}\geq r_{jk}/x_{jk}$ for all $i\rightarrow j$, $j\rightarrow k$, and\\
	$\hat{Q}_{ij}(\overline{q})\leq0$, $\underline{v}_i-2r_{ij}\hat{P}_{ij}^+(\overline{p})>0$ for all $i\rightarrow j$.
	\end{itemize}
	\end{theorem}

Now we prove Lemma \ref{thm: recover} through the primal construction method. For brevity, we prove Lemma \ref{thm: recover} for one-line networks as in Figure \ref{figure: one line}, where the notations can be simplified as in Appendix \ref{app: notation}. Generalization of the proof to tree networks is the same as that in Appendix \ref{app: tree}, and omitted for brevity.

With the simplified notations, Lemma \ref{thm: recover} is rephrased as
	\begin{lemma}\label{lemma: recover}
	Consider a one-line network. Assume $f_0$ is strictly increasing and there exists $\overline{p}_i$ and $\overline{q}_i$ such that \eqref{relax constraints} holds for $i\in\hN^+$. Then the SOCP relaxation is exact if $\overline{v}_i=\infty$ for $i\in \hN^+$ and any one of the following conditions holds.
	\begin{itemize}
	\item[(i)] $\hat{P}_i(\overline{p})\leq0$, $\hat{Q}_i(\overline{q})\leq0$ for $i\in \hN^+$.
	\item[(ii)] $r_i/x_i=r_{i+1}/x_{i+1}$ for $i=1,\ldots,n-1$, and\\ $\underline{v}_i-2r_i\hat{P}_i^+(\overline{p})-2x_i\hat{Q}_i^+(\overline{q})>0$ for $i\in \hN^+$.
	\item[(iii)] $r_i/x_i\leq r_{i+1}/x_{i+1}$ for $i=1,\ldots,n-1$, and\\
	$\hat{P}_i(\overline{p})\leq0$, $\underline{v}_i-2x_i\hat{Q}_i^+(\overline{q})>0$ for $i\in \hN^+$.
	\item[(iv)] $r_i/x_i\geq r_{i+1}/x_{i+1}$ for $i=1,\ldots,n-1$, and\\
	$\hat{Q}_i(\overline{q})\leq0$, $\underline{v}_i-2r_i\hat{P}_i^+(\overline{p})>0$ for $i\in \hN^+$.
	\end{itemize}
	\end{lemma}

The following claim forms the basis of the proof of Lemma \ref{lemma: recover}. Its proof is the same as that of Corollary \ref{corollary: key}, and omitted for brevity.
\begin{claim}\label{claim: key}
	Consider a one-line network. Assume that $f_0$ is strictly increasing and $\overline{v}_i=\infty$ for $i\in \hN^+$. If every optimal solution $w$ of the SOCP relaxation satisfies \eqref{key step}, then the SOCP relaxation is exact.
	\end{claim}

According to Claim \ref{claim: key}, to prove Lemma \ref{lemma: recover}, it suffices to prove that \eqref{key step} holds for  every optimal solution $w$ of the SOCP relaxation if any one of the conditions in Lemma \ref{lemma: recover} holds. This is established in Claim \ref{claim: no reverse}--\ref{claim: Q}, which completes the proof of Lemma \ref{lemma: recover}.

	\begin{claim}\label{claim: no reverse}
	Consider a one-line network. Assume that there exists $\overline{p}_i$ and $\overline{q}_i$ such that \eqref{relax constraints} holds for $i\in\hN^+$. Then \eqref{key step} holds for every feasible point $w$ of the SOCP relaxation if $\hat{P}_i(\overline{p})\leq0$, $\hat{Q}_i(\overline{q})\leq0$ for $i\in \hN^+$.
	\end{claim}
	
	\begin{proof}
	If $\hat{P}_i(\overline{p})\leq0$, $\hat{Q}_i(\overline{q})\leq0$ for $i=1,\ldots,n$, then C1 holds, and it follows from Lemma \ref{lemma: main condition} that \eqref{key step} holds for every feasible point $w$ of the SOCP relaxation.
	\end{proof}

	\begin{claim}\label{claim: uniform}
	Consider a one-line network. Assume that there exists $\overline{p}_i$ and $\overline{q}_i$ such that \eqref{relax constraints} holds for $i\in\hN^+$. Then \eqref{key step} holds for every feasible point $w$ of the SOCP relaxation if $r_i/x_i=r_{i+1}/x_{i+1}$  for $i=1,\ldots,n-1$, and $\underline{v}_i-2r_i\hat{P}_i^+(\overline{p})-2x_i\hat{Q}_i^+(\overline{q})>0$ for $i\in \hN^+$.
	\end{claim}
	
	\begin{proof}
	For a feasible point $w$ of the SOCP relaxation, define
	$$\begin{pmatrix}\alpha_{ji}(w)\\ \beta_{ji}(w)\end{pmatrix}:=A_j(w)\cdots A_{i-1}(w)u_i$$
for $1\leq j\leq i\leq n$. Claim \ref{claim: uniform} is equivalent to $\alpha_{ji}(w)>0$ and $\beta_{ji}(w)>0$ for $1\leq j\leq i\leq n$ and all feasible $w$ to the SOCP relaxation. Fix an arbitrary feasible $w$ and an arbitrary $i\in\{1,\ldots,n\}$, it suffices to prove that $\alpha_{ji}(w)>0$ and $\beta_{ji}(w)>0$ for $j=1,\ldots,i$. In the rest of the proof, we abbreviate $\alpha_{ji}(w)$ and $\beta_{ji}(w)$ by $\alpha_{ji}$ and $\beta_{ji}$.

We prove that $\alpha_{ji}>0$ and $\beta_{ji}>0$ for $j=1,\ldots,i$ by mathematical induction on $j$. In particular, we prove the following hypothesis
$$\textbf{H1: }\alpha_{ji}>0, ~ \beta_{ji}>0, ~ \alpha_{ji}/\beta_{ji}=r_1/x_1 $$
inductively for $j=i,i-1,\ldots,1$. For brevity, define $\eta:=r_1/x_1$ and note that $r_j/x_j=\eta$ for $j\in \hN^+$ if $r_k/x_k=r_{k+1}/x_{k+1}$ for $k=1,\ldots,n-1$.

	\begin{itemize}
	\item If $j=i$, then $\alpha_{ji}=r_i$, $\beta_{ji}=x_i$, $\alpha_{ji}/\beta_{ji}=\eta$. The Hypothesis H1 holds.
	\item Assume that Hypothesis H1 holds for $j=k$ ($2\leq k\leq i$). When $j=k-1$, it can be verified that
	\begin{eqnarray*}
	\begin{pmatrix}
	\alpha_{ji}\\ \beta_{ji}
	\end{pmatrix} &=&
	\begin{pmatrix}
	1-\frac{2r_jP_j^+}{v_j} & -\frac{2r_jQ_j^+}{v_j} \\ -\frac{2x_jP_j^+}{v_j} & 1-\frac{2x_jQ_j^+}{v_j}
	\end{pmatrix}
	\begin{pmatrix}
	\alpha_{ki} \\ \beta_{ki}
	\end{pmatrix}\\
	&=&
	\frac{1}{v_j}\left(v_j-2r_jP_j^+ -2x_jQ_j^+\right)
	\begin{pmatrix}
	\alpha_{ki} \\ \beta_{ki}
	\end{pmatrix},
	\end{eqnarray*}
	therefore $\alpha_{ji}/\beta_{ji}=\alpha_{ki}/\beta_{ki}=\eta$. Besides, $\alpha_{ji}>0$, $\beta_{ji}>0$ since
	\begin{eqnarray*}
	&& v_j-2r_jP_j^+ -2r_jQ_j^+ \\
	&\geq& \underline{v}_j -2r_j\hat{P}_j^+(\overline{p}) -2x_j\hat{Q}_j^+(\overline{q})>0.
	\end{eqnarray*}
	Hence, Hypothesis H1 holds for $j=k-1$.
	\end{itemize}
It follows that Hypothesis H1 holds for $1\leq j\leq i$, which completes the proof of Claim \ref{claim: uniform}.
	\end{proof}
	
	\begin{claim}\label{claim: P}
	Consider a one-line network. Assume that there exists $\overline{p}_i$ and $\overline{q}_i$ such that \eqref{relax constraints} holds for $i\in\hN^+$. Then \eqref{key step} holds for every feasible point $w$ of the SOCP relaxation if $r_i/x_i\leq r_{i+1}/x_{i+1}$ for $i=1,\ldots,n-1$, and $\hat{P}_i(\overline{p})\leq0$, $\underline{v}_i-2x_i\hat{Q}_i^+(\overline{q})>0$ for $i\in \hN^+$.
	\end{claim}
	
	\begin{proof}
	As discussed in the proof of Claim \ref{claim: uniform}, it suffices to prove that for an arbitrary feasible point $w$ of the SOCP relaxation and an arbitrary $i\in\{1,\ldots,n\}$, we have $\alpha_{ji}(w)>0$, $\beta_{ji}(w)>0$ for $j=1,\ldots,i$.

We prove that $\alpha_{ji}>0$ and $\beta_{ji}>0$ for $j=1,\ldots,i$ by mathematical induction on $j$. In particular, we prove the following hypothesis
$$\textbf{H2: }\alpha_{ji}>0, ~ \beta_{ji}>0, ~ \alpha_{ji}/\beta_{ji}\geq r_i/x_i $$
inductively for $j=i,i-1,\ldots,1$. To start, note that $P_j^+=0$ for $j\in \hN^+$ since $0\leq P_j^+\leq \hat{P}_j^+(p )\leq \hat{P}_j^+(\overline{p})=0$.

	\begin{itemize}
	\item If $j=i$, then $\alpha_{ji}=r_i$, $\beta_{ji}=x_i$, $\alpha_{ji}/\beta_{ji}=r_i/x_i$. Hypothesis H2 holds.
	\item Assume that Hypothesis H2 holds for $j=k$ ($2\leq k\leq i$). When $j=k-1$, we have
	\begin{eqnarray*}
	\begin{pmatrix}
	\alpha_{ji} \\ \beta_{ji}
	\end{pmatrix} &=&
	\begin{pmatrix}
	1 & -\frac{2r_jQ_j^+}{v_j} \\ 0 & 1-\frac{2x_jQ_j^+}{v_j}
	\end{pmatrix}
	\begin{pmatrix}
	\alpha_{ki} \\ \beta_{ki}
	\end{pmatrix}.
	\end{eqnarray*}
Hence,
	\begin{eqnarray*}
	\beta_{ji} &=& \frac{1}{v_j}\left(v_j-2x_jQ_j^+\right) \beta_{ki}\\
	&\geq& \frac{1}{v_j} \left(\underline{v}_j-2x_j\left[Q_j^\mathrm{lin}(\bar{q})\right]^+\right) \beta_{ki}>0.
	\end{eqnarray*}
Then,
	\begin{eqnarray*}
	\alpha_{ji} &=& \alpha_{ki} - \frac{2r_jQ_j^+}{v_j}\beta_{ki} \geq \left(\frac{r_i}{x_i}- \frac{2r_jQ_j^+}{v_j}\right)\beta_{ki} \\
    &=& \frac{r_i}{x_i} \left(1- \frac{r_j/x_j}{r_i/x_i}\frac{2x_jQ_j^+}{v_j}\right)\beta_{ki}\\
    &\geq& \frac{r_i}{x_i} \left(1- \frac{2x_jQ_j^+}{v_j}\right)\beta_{ki}
	= \frac{r_i}{x_i}\beta_{ji} > 0.
	\end{eqnarray*}
Hence, Hypothesis H2 holds for $j=k-1$.
	\end{itemize}
	It follows that Hypothesis H2 holds for $1\leq j\leq i$, which completes the proof of Claim \ref{claim: P}.
	\end{proof}	

	\begin{claim}\label{claim: Q}
	Consider a one-line network. Assume that there exists $\overline{p}_i$ and $\overline{q}_i$ such that \eqref{relax constraints} holds for $i\in\hN^+$. Then \eqref{key step} holds for every feasible point $w$ of the SOCP relaxation if $r_i/x_i\geq r_{i+1}/x_{i+1}$ for $i=1,\ldots,n-1$, and $\hat{Q}_i(\overline{q})\leq0$, $\underline{v}_i-2r_i\hat{P}_i^+(\overline{p})>0$ for $i\in \hN^+$.
	\end{claim}
	
	\begin{proof}
	The proof of Claim \ref{claim: Q} is similar to that of Claim \ref{claim: P}, and omitted for brevity.
	\end{proof} 

\end{document}